\documentclass[oneside,english]{amsart}
\usepackage{amsthm}
\usepackage{amstext}
\usepackage{amssymb}

\makeatletter
\numberwithin{equation}{section} 
\numberwithin{figure}{section} 
\theoremstyle{plain}
\theoremstyle{plain}
\newtheorem{thm}{Theorem}
  \theoremstyle{definition}
  \newtheorem{defn}[thm]{Definition}
  \theoremstyle{plain}
  \newtheorem{lem}[thm]{Lemma}
  \theoremstyle{plain}
  \newtheorem{prop}[thm]{Proposition}
  \theoremstyle{plain}
  \newtheorem{cor}[thm]{Corollary}

\numberwithin{equation}{section}
\numberwithin{thm}{section}

\makeatother

\usepackage{babel}

\begin{document}

\title{Confluent Operator Algebras and the Closability Property}

\author{H. Bercovici, R. G. Douglas, C. Foias, and C. Pearcy}

\address{HB: Department of Mathematics, Indiana University, Bloomington, IN
47405}

\email{bercovic@indiana.edu}

\address{RGD, CF, and CP: Department of Mathematics, Texas A\&M University,
College Station, TX 77843}

\email{rdouglas@math.tamu.edu\\
foias@math.tamu.edu\\
pearcy@math.tamu.edu}

\dedicatory{This paper is dedicated to the memory of our good friends and mentors,
Paul R. Halmos and B\'ela Sz\H{o}kefalvi-Nagy.}

\thanks{HB and RGD were supported in part by grants from the National Science
Foundation.}


\begin{abstract}
Certain operator algebras $\mathcal{A}$ on a Hilbert space have the
property that every densely defined linear transformation commuting
with $\mathcal{A}$ is closable. Such algebras are said to have the
\emph{closability property}. They are important in the study of the
transitive algebra problem. More precisely, if $\mathcal{A}$ is a
two-transitive algebra with the closability property, then $\mathcal{A}$
is dense in the algebra of all bounded operators, in the weak operator
topology. In this paper we focus on algebras generated by a completely
nonunitary contraction, and produce several new classes of algebras
with the closability property. We show that this property follows
from a certain strict cyclicity property, and we give very detailed
information on the class of completely nonunitary contractions satisfying
this property, as well as a stronger property which we call \emph{confluence}.
\end{abstract}
\maketitle

\section{Introduction\label{sec:Introduction}}

\markboth{}{}Probably the best known problem in operator theory is
the question of whether every bounded linear operator on a complex,
separable, infinite dimensional Hilbert space $\mathcal{H}$ has a
nontrivial invariant subspace. Despite considerable effort by many
researchers for more than half a century, the general problem remains
open. A generalization, still unresolved, asks whether every transitive
algebra of operators must be dense in the weak operator topology.
(Recall an algebra is said to be transitive if there are no common
nontrivial invariant subspaces for the operators in it.)

In the sixties, Arveson approached this problem iteratively, starting
from an observation going back essentially to von Neumann. Namely,
assume that $\mathcal{A}$ is an algebra of operators on a Hilbert
space $\mathcal{H}$, and $n\ge1$ is an integer. The algebra $\mathcal{A}$
is said to be $n$-transitive if every invariant subspace for \[
\mathcal{A}^{(n)}=\{X^{(n)}=\underbrace{X\oplus X\oplus\cdots\oplus X}_{n\text{ times}}:X\in\mathcal{A}\}\]
is invariant for every operator of the form $Y^{(n)}$ where $Y$
is an operator on $\mathcal{H}$. Then $\mathcal{A}$ is dense, in
the weak operator topology, if and only if it is $n$-transitive for
every $n\ge1$. Arveson observed that 2-transitivity is equivalent
to the following statement: every \emph{closed} linear transformation
commuting with $\mathcal{A}$ is a scalar multiple of the identity
operator. For $n\ge3$, $n$-transitivity is implied by a similar
statement for \emph{densely defined} linear transformations commuting
with $\mathcal{A}$. Thus, provided that every densely defined linear
transformation commuting with $\mathcal{A}$ is closable, 2-transitivity
implies $n$-transitivity for all $n$. As a consequence, Arveson
was able to prove that  transitive algebras containing certain kinds
of subalgebras are indeed dense in the weak operator topology. His
results apply to algebras on an $L^{2}$-space, containing the algebra
$L^{\infty}$ of all bounded measurable multipliers, or on the Hardy
space $H^{2}(\mathbb{D})$, containing the algebra $H^{\infty}(\mathbb{D})$.
A few similar results were obtained by others for closely related
algebras in the following years; see for instance \cite[Chapter 8]{rad-ros}.

A year ago, Haskell Rosenthal became interested in the question of
which algebras of operators on Hilbert space had what he called the
\emph{closability property} which means that every densely defined
linear transformation in its commutant is closable. A key step in
Arveson's proofs was to show that the algebras $L^{\infty}$ acting
on $L^{2}$, and $H^{\infty}(\mathbb{D})$ acting on $H^{2}(\mathbb{D})$,
have the closability property. Rosenthal showed that various algebras
have the closability property and asked the authors a specific followup
question. In finding the answer, the question piqued our interest
which resulted in a series of questions related to this topic. Our
investigation took us in some unexpected directions, making surprising
connections with other topics in operator theory.

After some preliminaries in Section \ref{sec:Preliminaries}, we begin
in Section \ref{sec:The-Closability-Property} by investigating the
closability property and determining some algebras which have it,
as well as some that do not. In Section \ref{sec:Rationally-Strictly-Cyclic}
we introduce the concept of a rationally strictly cyclic vector, and
show that the existence of such a vector for a commutative algebra
$\mathcal{A}$ implies the closability property. In Section \ref{sec:Quasisimilar-Algebras}
we discuss the invariance of the closability property, and of the
existence of rationally strictly cyclic vectors, under an appropriate
notion of quasisimilarity. We deduce, for instance, that the commutant
of any contraction of class $C_{0}$ has the closability property.
In the course of our study, the importance of something like a functional
calculus for quotients became clear. To make this idea precise, in
Section \ref{sec:Confluent-Algebras-of} we study the related notion
of confluence (introduced in Section \ref{sec:Rationally-Strictly-Cyclic})
as it applies to the algebra obtained by applying the $H^{\infty}$
functional calculus to a completely nonunitary contraction. Confluence
implies the existence of a rationally strictly cyclic vector, and
therefore the closability property as well. Section \ref{sec:Confluence-and-Functional-mod}
contains a thorough study of confluence in the context of functional
models for contractions. In particular, a characterization is obtained
for those contractions which are quasisimilar to the unilateral shift
of multiplicity one. This characterization involves the `size' of
the analytic functions in the reproducing kernel representative for
the operator.

The analysis of confluence is somewhat subtle and rests on the harmonic
analysis of contractions \cite{Sz-N-F}, the theory of the class $C_{0}$
\cite{Berco}, the theory of dual algebras \cite{B-F-P}, and results
about the class $\mathcal{B}_{1}({\mathbb{D}})$ \cite{C-D}.

We thank Haskell Rosenthal for the questions which led to this research.

\section{Preliminaries\label{sec:Preliminaries}}

We will work with operators on Hilbert spaces over the complex numbers
$\mathbb{C}$. The algebra of bounded linear operators on a Hilbert
space $\mathcal{H}$ is denoted $\mathcal{L}(\mathcal{H})$. Given
$T\in\mathcal{L}(\mathcal{H})$, $\mathcal{P}_{T}$ denotes the smallest
unital algebra containing $T$, that is the set of all polynomials
in $T$. The closure of $\mathcal{P}_{T}$ in the weak operator topology
(also known as WOT) is denoted $\mathcal{W}_{T}$. The norm closure
of a subset $\mathcal{M}\subset\mathcal{H}$ is denoted $\overline{\mathcal{M}}$.
The orthogonal projection onto a closed linear subspace $\mathcal{M}\subset\mathcal{H}$
is denoted $P_{\mathcal{M}}$.

Several special operators play an important role. The space $L^{2}$
is the space of functions defined on the unit circle $\mathbb{T}$
which are square integrable relative to arclength measure. The bilateral
shift $U\in\mathcal{L}(L^{2})$ is the unitary operator defined by
$(Uf)(\zeta)=\zeta f(\zeta)$ for $f\in L^{2}$ and a.e. $\zeta\in\mathbb{T}$.
The Hardy space $H^{2}\subset L^{2}$ is the cyclic space for $U$
generated by the constant function $1$, and $S\in\mathcal{L}(H^{2})$
is the unilateral shift of multiplicity 1 defined as $S=U|H^{2}$.
More generally, denote by $H^{\infty}=H^{\infty}(\mathbb{D})$, the
algebra of bounded analytic functions in the unit disk $\mathbb{D}$.
For each $u\in H^{\infty}$ one defines an analytic Toeplitz operator
$T_{u}\in\mathcal{L}(H^{2})$ as the operator of pointwise multiplication
by $u$. Here one takes advantage of the fact that functions in $H^{\infty}$
have a.e. defined radial limits on $\mathbb{T}$.

Given a subset $\mathcal{A}\subset\mathcal{L}(\mathcal{H})$, $\mathcal{A}'$
denotes the set of operators commuting with every element of $\mathcal{A}$.
The set $\mathcal{A}'$ is called the commutant of $\mathcal{A}$,
and it is an algebra, closed in the weak operator topology. An important
example is\[
\{S\}'=\mathcal{W}_{S}=\{T_{u}:u\in H^{\infty}\}.\]

A function $m\in H^{\infty}$ is inner if $|m(\zeta)|=1$ for a.e.
$\zeta\in\mathbb{T}$. For every inner function $m\in H^{\infty}$,
the space $mH^{2}=T_{m}H^{2}$ is closed and invariant for $S$. The
compression of $S$ to $\mathcal{H}(m)=H^{2}\ominus mH^{2}$ is denoted
$S(m)$. In other words, $S(m)=P_{\mathcal{H}(m)}S|\mathcal{H}(m)$
or, equivalently, $S(m)^{*}=S^{*}|\mathcal{H}(m)$. Another important
example of a commutant is\[
\{S(m)^{*}\}'=\mathcal{W}_{S(m)^{*}}=\{T_{u}^{*}|\mathcal{H}(m):u\in H^{\infty}\}.\]
This was proved by Sarason.

An operator $T\in\mathcal{L}(\mathcal{H})$ is a contraction if $\|T\|\le1$.
A contraction $T$ is completely nonunitary if it has no invariant
subspace on which it acts as a unitary operator. For completely nonunitary
contractions $T$, there is a homomorphism $u\mapsto u(T)\in\mathcal{L}(\mathcal{H})$
which extends the polynomial functional calculus to functions $u\in H^{\infty}$.
This is called the Sz.-Nagy---Foias functional calculus. For instance,
$u(S)=T_{u}$, and $u(S(m))=P_{\mathcal{H}(m)}T_{u}|\mathcal{H}(m)$.

A completely nonunitary contraction $T\in\mathcal{L}(\mathcal{H})$
is of class $C_{0}$ if $u(T)=0$ for some $u\in H^{\infty}\setminus\{0\}$.
The ideal $\{u\in H^{\infty}:u(T)=0\}\subset H^{\infty}$ is principal,
and it is generated by an inner function, uniquely determined up to
a constant factor of absolute value $1$. This function is called
the minimal function of $T$. The most basic example is $S(m)$, whose
minimal function is $m$.

We refer the reader to \cite{Sz-N-F} for further background on the
analysis of contractions, to \cite{B-F-P} for dual algebras, and
to \cite{Berco} for the class $C_{0}$. We will refer as needed to
these and other original sources for specific results.

\section{The Closability Property\label{sec:The-Closability-Property}}

Consider a unital subalgebra $\mathcal{A}$ of the algebra $\mathcal{L}(\mathcal{H})$
of bounded operators on the complex Hilbert space $\mathcal{H}$.
The algebra $\mathcal{A}$ is not assumed to be norm closed.
\begin{defn}
\label{def:commuting-unbounded}A linear transformation $X\colon\mathcal{D}(X)\to\mathcal{H}$
is said to \emph{commute} with $\mathcal{A}$ if for every $h\in\mathcal{D}(X)$
and every $T\in\mathcal{A}$ we have $Th\in\mathcal{D}(X)$ and \[
XTh=TXh.\]

\end{defn}
We define now the main concept we study in this paper.
\begin{defn}
\label{def:closability}The algebra $\mathcal{A}$ is said to have
the \emph{closability property} if every linear transformation $X$
which commutes with $\mathcal{A}$, and whose domain $\mathcal{D}(X)$
is dense in $\mathcal{H}$, is closable.
\end{defn}
We recall that a linear transformation $X$ is \emph{closable} if
the closure of its graph\[
\mathcal{G}(X)=\{h\oplus Xh\colon h\in\mathcal{D}(X)\}\]
is again the graph of a linear transformation, usually denoted $\overline{X}$
and called the closure of $X$. Equivalently, given a sequence $h_{n}\in\mathcal{D}(X)$
such that $\lim_{n\to\infty}\|h_{n}\|=0$ and the limit $k=\lim_{n\to\infty}Xh_{n}$
exists, it follows that $k=0$.

The following observation is a trivial consequence of the fact that
a linear transformation commuting with an algebra also commutes with
smaller algebras.
\begin{lem}
\label{lem:subalg-and-commutant}Assume that $\mathcal{A}\subset\mathcal{B}\subset\mathcal{L}(\mathcal{H})$
are unital algebras. If $\mathcal{A}$ is has the closability property
then so does $\mathcal{B}$. In particular, if $\mathcal{A}$ is a
commutative and has the closability property, then its commutant $\mathcal{A}'$
also has the closability property.
\end{lem}
We start with some examples of algebras which do not have the closability
property. The arguments are based on the following simple fact.
\begin{lem}
\label{lem:two-invariant-manifolds}Let $\mathcal{A}$ be a unital
subalgebra of $\mathcal{L}(\mathcal{H})$. Assume that there exist
linear manifolds $\mathcal{M},\mathcal{N}\subset\mathcal{H}$ such
that
\begin{enumerate}
\item $T\mathcal{M}\subset\mathcal{M}$ and $T\mathcal{N}\subset\mathcal{N}$
for every $T\in\mathcal{A}$;
\item $\mathcal{M}\cap\mathcal{N}=\{0\}$ and $\overline{\mathcal{M}+\mathcal{N}}=\mathcal{H}$;
\item $\overline{\mathcal{M}}\cap\overline{\mathcal{N}}\ne\{0\}$.
\end{enumerate}
Then $\mathcal{A}$ does not have the closability property.

\end{lem}
\begin{proof}
Define a linear transformation with domain $\mathcal{D}(X)=\mathcal{M}+\mathcal{N}$
by setting $Xh=0$ for $h\in\mathcal{M}$ and $Xh=h$ for $h\in\mathcal{N}$.
If $X$ were closable, its closure would satisfy $\overline{X}h=0$
and $\overline{X}h=h$ for any $h\in\overline{\mathcal{M}}\cap\overline{\mathcal{N}}$,
and this is absurd for $h\ne0$.\end{proof}
\begin{prop}
\label{pro:nonCPalgebras}The following algebras do not have the closability
property:
\begin{enumerate}
\item The algebra $\mathcal{P}_{S}$ generated by the unilateral shift $S$.
\item The algebra $\mathcal{P}_{S(m)}$, where $m$ is an inner function
which is not a finite Blaschke product.
\item The WOT-closed algebra $\mathcal{W}_{S^{*}}$.
\item The WOT-closed algebra $\mathcal{W}_{U}$ generated by the bilateral
shift $U$ on $L^{2}$.
\item Any algebra of the form $\mathcal{A}\otimes I_{\mathcal{K}}$, where
$\mathcal{A}\subset\mathcal{L}(\mathcal{H})$ is a unital algebra,
and $\mathcal{K}$ is an infinite dimensional Hilbert space.
\end{enumerate}
\end{prop}
\begin{proof}
For the first example, choose an outer function $f\in H^{2}$ which
is not rational, and define $\mathcal{M}$ to consist of all polynomials
and $\mathcal{N}=\{pf:p\text{ a polynomial}\}$. The hypotheses of
Lemma \ref{lem:two-invariant-manifolds} are verified trivially since
both of these spaces are dense in $H^{2}$.

Next, assume that $m$ is an inner function but not a finite Blaschke
product, and consider a factorization $m=m_{1}m_{2}$ such that the
inner functions $m_{j}$ are not finite Blaschke products. We can
define then subspaces $\mathcal{M},\mathcal{N}\subset\mathcal{H}(m)$
by $\mathcal{M}=\{P_{\mathcal{H}(m)}p:p\text{ a polynomial}\}$ and
$\mathcal{N}=\{P_{\mathcal{H}(m)}(pm_{2}):p\text{ a polynomial}\}$.
The space $\mathcal{M}$ is dense in $\mathcal{H}(m)$, so to verify
the hypotheses of Lemma \ref{lem:two-invariant-manifolds} it suffices
to show that $\mathcal{M}\cap\mathcal{N}=\{0\}$. Consider indeed
two polynomials $p,q$ such that $P_{\mathcal{H}(m)}p=P_{\mathcal{H}(m)}(qm_{2})$.
In other words, we have $p=qm_{2}+rm_{1}m_{2}$ for some $r\in H^{2}$.
If $p\ne0$, this equality implies that the inner factor of $p$ (obviously
a finite Blaschke product) is divisible by $m_{2}$, contrary to our
choice of factors.

For example (3), we choose $\mathcal{M}=\{p:p\text{ a polynomial}\}\subset H^{2}$,
and we denote by $\mathcal{N}$ the linear manifold generated by the
functions $k_{\lambda}(z)=(1-\lambda z)^{-1}$, $\lambda\in\mathbb{D}\setminus\{0\}$.
These spaces are dense in $H^{2}$, and the identities\[
(S^{*}p)(z)=\frac{p(z)-p(0)}{z},\quad S^{*}k_{\lambda}=\lambda k_{\lambda}\]
easily imply that they are invariant under $\mathcal{W}_{S^{*}}$.
Finally, a function $p$ in their intersection must be both a polynomial,
and a rational function vanishing at $\infty$, hence $p=0$.

For example (4), define two sets $\omega_{\pm}=\{e^{\pm it}:0<t<3\pi/2\}\subset\mathbb{T}$,
denote by $\chi_{\pm}$ their characteristic functions, and set $\mathcal{M}=\chi_{+}H^{2}$
and $\mathcal{N}=\chi_{-}H^{2}$. Since $\overline{\mathcal{M}}=\chi_{+}L^{2}$
and $\overline{\mathcal{N}}=\chi_{-}L^{2}$, we clearly have $\overline{\mathcal{M}+\mathcal{N}}=L^{2}$
and $\overline{\mathcal{M}}\cap\overline{\mathcal{N}}=\chi_{\omega_{+}\cap\omega_{-}}L^{2}$.
The fact that $\mathcal{M}\cap\mathcal{N}=\{0\}$ follows easily from
the F. and M. Riesz theorem.

Finally, assume that $\mathcal{K}$ is an infinite dimensional Hilbert
space, and let $\mathcal{M}_{0},\mathcal{N}_{0}\subset\mathcal{K}$
be two dense linear manifolds such that $\mathcal{M}_{0}\cap\mathcal{N}_{0}=\{0\}$.
Then $\mathcal{M}=\mathcal{H}\otimes\mathcal{M}_{0}$ and $\mathcal{N}=\mathcal{H}\otimes\mathcal{N}_{0}$
will satisfy the hypotheses of Lemma \ref{lem:two-invariant-manifolds}
for the algebra $\mathcal{A}\otimes I_{\mathcal{K}}$.
\end{proof}
The first two examples above indicate that an algebra with the closability
property must be reasonably large, while the last one shows that it
should not have uniform infinite multiplicity. In this paper we will
focus on algebras which have multiplicity one. The first example of
an algebra with the closability property was of this kind: any maximal
abelian selfadjoint subalgebra of $\mathcal{L}(\mathcal{H})$ has
the closability property as shown in \cite{arv}. This, along with
the examples described in the following proposition (the first of
which also appeared in \cite{arv}), will be treated in a unified
manner in Section \ref{sec:Rationally-Strictly-Cyclic}. The proofs
of these particular cases do in fact suggest the more general methods.
\begin{prop}
\label{pro:SandS(m)haveCP}The algebras $\mathcal{W}_{S}$ and $\mathcal{W}_{S(m)}$
have the closability property.\end{prop}
\begin{proof}
Recall first that every function in $H^{2}$ is the quotient of two
bounded functions in $H^{\infty}$. For instance, given a nonzero
function $f\in H^{2}$, denote by $v_{f}$ the unique outer function
defined by the requirements that $v_{f}(0)>0$ and\[
|v_{f}(\zeta)|=\min\left\{ 1,\frac{1}{|f(\zeta)|}\right\} \]
for almost every $\zeta\in\mathbb{T}$. The functions $v_{f}$ and
$u_{f}=fv_{f}$ belong to $H^{\infty}$, and in fact\begin{equation}
|u_{f}(\zeta)|=\min\{1,|f(\zeta)|\}\quad\text{a.e. }\zeta\in\mathbb{T}.\label{eq:d-f-n-f-less-than1}\end{equation}

Consider first the algebra $\mathcal{W}_{S}$ which consists precisely
of the analytic Toeplitz operators $T_{u}$ with $u\in H^{\infty}$.
Let $X$ be a densely defined linear transformation commuting with
this algebra, and let $f,g\in\mathcal{D}(X)$. Observe first that
$u_{f}=v_{f}f=T_{v_{f}}f\in\mathcal{D}(X)$, and therefore we can
write\begin{eqnarray*}
v_{g}u_{f}Xg & = & T_{v_{g}u_{f}}Xg=XT_{v_{g}u_{f}}g=X(v_{g}u_{f}g)=X(u_{f}u_{g})\\
 & = & XT_{u_{g}}u_{f}=T_{u_{g}}Xu_{f}=u_{g}Xu_{f}.\end{eqnarray*}
Let now $g_{n}\in\mathcal{D}(X)$ be a sequence converging to zero
such that the limit $h=\lim_{n\to\infty}Xh_{n}$ exists. Passing if
necessary to a subsequence, we may assume that $g_{n}(\zeta)\to0$
for almost every $\zeta\in\mathbb{T}$. By virtue of (\ref{eq:d-f-n-f-less-than1})
we also have $|v_{g_{n}}(\zeta)|\to1$ and $u_{g_{n}}(\zeta)\to0$
for a.e. $\zeta$, and therefore\[
\|u_{g_{n}}Xu_{f}\|^{2}=\frac{1}{2\pi}\int_{0}^{2\pi}|u_{g_{n}}(e^{it})|^{2}|(Xu_{f})(e^{it})|^{2}\, dt\to0\]
as $n\to\infty$ by the dominated convergence theorem. The identity\[
v_{g_{n}}u_{f}Xg_{n}=u_{g_{n}}Xu_{f}\]
proved earlier, along with the fact that $|v_{g_{n}}|\to1$ a.e.,
implies that $u_{f}h=0$ for every $f\in\mathcal{D}(X)$. Choosing
a nonzero function $f$ we deduce that $h=0$, thus proving that $X$
is closable.

Consider now a densely defined linear transformation $X$ commuting
with $\mathcal{W}_{S(m)}$. Given $f\in\mathcal{D}(X)$, the vector
$P_{\mathcal{H}(m)}u_{f}=P_{\mathcal{H}(m)}(v_{f}f)=v_{f}(S(m))f$
belongs to $\mathcal{D}(X)$. If $g$ is another vector in $\mathcal{D}(X)$,
we have \begin{eqnarray*}
(v_{g}u_{f})(S(m))g & = & P_{\mathcal{H}(m)}(v_{g}u_{f}g)\\
 & = & P_{\mathcal{H}(m)}(u_{f}u_{g})=u_{g}(S(m))P_{\mathcal{H}(m)}u_{f},\end{eqnarray*}
and therefore\[
X(v_{g}u_{f})(S(m))g=u_{g}(S(m))XP_{\mathcal{H}(m)}u_{f}=P_{\mathcal{H}(m)}(u_{g}XP_{\mathcal{H}(m)}u_{f}).\]
Thus we obtain\[
(v_{g}u_{f})(S(m))Xg=P_{\mathcal{H}(m)}(u_{g}XP_{\mathcal{H}(m)}u_{f}),\]

\[
u_{f}(S(m))Xg=v_{g}(S(m))P_{\mathcal{H}(m)}(u_{g}XP_{\mathcal{H}(m)}u_{f}),\]
and finally\[
\|u_{f}(S(m))Xg\|\le\|P_{\mathcal{H}(m)}(u_{g}XP_{\mathcal{H}(m)}u_{f})\|\le\|u_{g}XP_{\mathcal{H}(m)}u_{f}\|.\]
Consider now a sequence $g_{n}\in\mathcal{D}(X)$ such that $g_{n}\to0$
and $Xg_{n}\to h$. As in the case of $S$, the preceding inequality
implies that $u_{f}(S(m))h=0$ for every $f\in\mathcal{D}(X)$. Equivalently,
$m$ divides the function $u_{f}h$ for every $f\in\mathcal{D}(X)$.
Note now that $f$ and $u_{f}$ have the same inner factor, and therefore
$m$ divides $fh$ for every $f\in\mathcal{D}(X)$. Denote by $d$
the greatest common inner divisor of $\{f:f\in\mathcal{D}(X)\}$.
The density of $\mathcal{D}(X)$ implies that $d\wedge m=1$, and
therefore $m$ must divide $h$ by virtue of \cite[Lemma III.4.5]{Sz-N-F}.
In other words, $h$ is the zero vector in $\mathcal{H}(m)$, and
the desired conclusion that $X$ is closable follows.
\end{proof}
Note incidentally that the example of $\mathcal{W}_{S}$ shows that
closability is not generally inherited by the adjoint algebra.

We conclude this section with a simple fact which will be used in
the study of closability for quasisimilar algebras. Let $\mathcal{A}_{i}\subset\mathcal{L}(\mathcal{H}_{i})$,
$i\in I$, be algebras. The algebra $\bigoplus_{i\in I}\mathcal{A}_{i}\subset\mathcal{L}\left(\bigoplus_{i\in I}\mathcal{H}_{i}\right)$
consists of those operators of the form $\bigoplus_{i\in I}T_{i}$,
where $T_{i}\in\mathcal{A}_{i}$ for each $i$, and $\sup\{\|T_{i}\|:i\in I\}<\infty$.
\begin{lem}
\label{lem:CP-for-direct-sums}A direct sum $\mathcal{A}=\bigoplus_{i\in I}\mathcal{A}_{i}$
has the closability property if and only if $\mathcal{A}_{i}$ has
this property for every $i\in I$.\end{lem}
\begin{proof}
Assume first that $\mathcal{A}$ has the closability property, and
$X_{i_{0}}$ is a densely defined linear transformation on $\mathcal{H}_{i_{0}}$
commuting with $\mathcal{A}_{i_{0}}$ for some $i_{0}\in I$. We define
a linear transformation $X$ with dense domain $\mathcal{D}(X)=\bigoplus_{i\in I}\mathcal{D}_{i}$,
where $\mathcal{D}_{i_{0}}=\mathcal{D}_{i_{0}}$, $\mathcal{D}_{i}=\mathcal{H}_{i}$
for $i\ne i_{0}$, and $X\bigoplus h_{i}=\bigoplus k_{i}$, where
$k_{i_{0}}=X_{i_{0}}h_{i_{0}}$ and $k_{i}=0$ for $i\ne i_{0}$.
The linear transformation $X$ commutes with $\mathcal{A}$, hence
it is closable. It follows that $X_{i_{0}}$ must be closable as well.
Conversely, assume that each $\mathcal{A}_{i}$ is closable, and let
$X$ be a densely defined linear transformation commuting with $\mathcal{A}$.
If $P_{j}\in\mathcal{A}$ denotes the orthogonal projection onto the
$j$th component of $\bigoplus_{i\in I}\mathcal{H}_{i}$, we have
then $P_{j}X\subset XP_{j}$, and the linear transformation $X_{j}:\mathcal{D}_{j}=P_{j}\mathcal{D}(X)\to\mathcal{H}_{j}$
defined by $X_{j}=X|\mathcal{D}_{j}$ commutes with $\mathcal{A}_{j}$.
It follows that each $X_{j}$ is closable, and then it is easy to
verify that $X$ is closable as well.
\end{proof}

\section{Rationally Strictly Cyclic Vectors and Confluence\label{sec:Rationally-Strictly-Cyclic}}

The examples  in Proposition \ref{pro:SandS(m)haveCP}, as well as
maximal abelian selfadjoint subalgebras (also known as \emph{masa}s),
can actually be treated in a unified manner. For this purpose we need
a new concept.
\begin{defn}
\label{def:rat-cyc-vect}Let $\mathcal{A}\subset\mathcal{L}(\mathcal{H})$
be a unital algebra. A vector $h_{0}\in\mathcal{H}$ is called a \emph{rationally
strictly cyclic} vector for $\mathcal{A}$ if for every $h\in\mathcal{H}$
there exist $A,B\in\mathcal{A}$ such that $Bh=Ah_{0}$ and $\ker B=\{0\}$.
\end{defn}
Recall that $h_{0}$ is said to be \emph{strictly cyclic} for $\mathcal{A}$
if $\mathcal{A}h_{0}=\mathcal{H}$. Thus, a strictly cyclic vector
is rationally strictly cyclic, but not conversely. None of the examples
considered in this paper exhibit strictly cyclic vectors.
\begin{lem}
\label{lem:some-algebras-with-RSC}The following algebras have rationally
strictly cyclic vectors:
\begin{enumerate}
\item $\mathcal{W}_{S}$
\item $\mathcal{W}_{S(m)}$
\item Any masa on a separable Hilbert space. More generally, any masa with
a cyclic vector.
\end{enumerate}
\end{lem}
\begin{proof}
The vector $1\in H^{2}$ is rationally strictly cyclic for $\mathcal{W}_{S}$,
while $1-\overline{m(0)}m=P_{\mathcal{H}(m)}1$ is rationally strictly
cyclic for $\mathcal{W}_{S(m)}$. For (3), we may assume that $\mathcal{H}=L^{2}(\mu)$,
where $\mu$ is a Borel probability measure on some compact topological
space, and $\mathcal{A}=\{M_{u}:u\in L^{\infty}(\mu)\}$, where\[
M_{u}f=uf,\quad u\in L^{\infty}(\mu),f\in L^{2}(\mu).\]
Since every function in $L^{2}(\mu)$ is the quotient of two bounded
functions, the constant function $1$ is rationally strictly cyclic
for $\mathcal{A}$.
\end{proof}
Here are two useful properties of algebras with rationally strictly
cyclic vectors.
\begin{lem}
\label{lem:separating-for-commutant}Let $\mathcal{A}\subset\mathcal{L}(\mathcal{H})$
be a unital algebra with a rationally strictly cyclic vector $h_{0}$.
\begin{enumerate}
\item If $T\in\mathcal{A}'\setminus\{0\}$ then $Th_{0}\ne0$.
\item If $\mathcal{A}$ is commutative and $\mathcal{D}\subset\mathcal{H}$
is a dense linear manifold, invariant for $\mathcal{A}$, then\[
\bigcap\{\ker T:T\in\mathcal{A},Th_{0}\in\mathcal{D}\}=\{0\}.\]

\end{enumerate}
\end{lem}
\begin{proof}
Assume that $T\in\mathcal{A}'$ and $Th_{0}=0$. Given $x\in\mathcal{H}$,
choose $A_{x},B_{x}\in\mathcal{A}$ such that $B_{x}x=A_{x}h_{0}$
and $\ker B_{x}=\{0\}.$ We have then\[
B_{x}Tx=TB_{x}x=TA_{x}h_{0}=A_{x}Th_{0}=0,\]
and therefore $Tx=0$. This implies that $T=0$ since $x$ is arbitrary.

Assume now that $\mathcal{A}$ is commutative and $\mathcal{D}\subset\mathcal{H}$
is a dense linear manifold, invariant for $\mathcal{A}$. Let $h\in\mathcal{H}$
be a vector such that $Ah=0$ whenever $A\in\mathcal{A}$ and $Ah_{0}\in\mathcal{D}$.
Using the notation above, we have $A_{x}h_{0}=B_{x}x\in\mathcal{D}$
whenever $x\in\mathcal{D}$, and therefore $A_{x}h=0$ for $x\in\mathcal{D}$.
Thus\begin{eqnarray*}
0 & = & B_{h}A_{x}h=A_{x}B_{h}h=A_{x}A_{h}h_{0}=A_{h}A_{x}h_{0}\\
 & = & A_{h}B_{x}x=B_{x}A_{h}x\end{eqnarray*}
for $x\in\mathcal{D}$, which implies $A_{h}x=0$ for such vectors
$x$. From the density of $\mathcal{D}$ we deduce that $A_{h}=0$,
and thus $B_{h}h=A_{h}h_{0}=0$ and $h=0$, as desired.
\end{proof}
We can now prove a generalization of Proposition \ref{pro:SandS(m)haveCP}.
\begin{thm}
\label{thm:abelian+RSCimplies-CP}Any commutative algebra $\mathcal{A}$
with a rationally strictly cyclic vector has the closability property.\end{thm}
\begin{proof}
Let $h\in\mathcal{H}$ be a rationally strictly cyclic vector for
the algebra $\mathcal{A}\subset\mathcal{L}(\mathcal{H})$, and let
$X$ be a linear transformation with dense domain $\mathcal{D}(X)$,
commuting with $\mathcal{A}$. For every $x\in\mathcal{H}$ we choose
operators $A_{x},B_{x}\in\mathcal{A}$ satisfying $B_{x}x=A_{x}h_{0}$
and $\ker B_{x}=\{0\}$. Consider a sequence $x_{n}\in\mathcal{D}(X)$
such that $x_{n}\to0$ and $Xx_{n}\to h$ as $n\to\infty$. By Lemma
\ref{lem:separating-for-commutant}(2), it will suffice to show that
$Th=0$ whenever $T\in\mathcal{A}$ and $Th_{0}\in\mathcal{D}(X)$.
Observe first that for such operators $T$ we have\[
B_{Xx_{n}}TXx_{n}=XTB_{Xx_{n}}Xx_{n}=XTA_{Xx_{n}}h_{0}=A_{Xx_{n}}XTh_{0}.\]
Multiplying both sides by $B_{XTh_{0}}$ and using commutativity,
we obtain\begin{eqnarray*}
B_{Xx_{n}}B_{XTh_{0}}TXx_{n} & = & A_{Xx_{n}}XB_{XTh_{0}}Th_{0}=A_{Xx_{n}}XA_{XTh_{0}}h_{0}\\
 & = & A_{XTh_{0}}XA_{Xx_{n}}h_{0}=A_{XTh_{0}}XB_{Xx_{n}}x_{n}\\
 & = & B_{Xx_{n}}A_{XTh_{0}}Xx_{n},\end{eqnarray*}
and therefore

\[
B_{XTh_{0}}TXx_{n}=A_{XTh_{0}}Xx_{n}\]
because $B_{Xx_{n}}$ is injective. Letting $n\to\infty$ we obtain
$B_{XTh_{0}}Th=0$ and hence $Th=0$, as desired.\end{proof}
\begin{cor}
There exists no $2$-transitive, commutative subalgebra of $\mathcal{L}(\mathcal{H})$
with a rationally strictly cyclic vector.
\end{cor}
The calculations in the preceding proof can be used to relate closed,
densely defined linear transformations commuting with $\mathcal{A}$
with linear transformations of the form $B^{-1}A$ with $A,B\in\mathcal{A}$
and $\ker B=\{0\}$. Note that\[
\mathcal{G}(B^{-1}A)=\{h\oplus k\in\mathcal{H}\oplus\mathcal{H}:Ah=Bk\},\]
and this is generally larger than\[
\mathcal{G}(AB^{-1})=\{Bh\oplus Ah:h\in\mathcal{H}\}.\]
Also observe that two linear transformations of this form, say $B^{-1}A,B_{1}^{-1}A_{1}$,
which agree on a dense linear manifold $\mathcal{D}$, must in fact
be equal. Indeed, the equality on $\mathcal{D}$ implies that $BA_{1}h=B_{1}Ah$
for $h\in\mathcal{D}$, and therefore $B_{1}A=BA_{1}$. Thus for $h\oplus k\in\mathcal{G}(B^{-1}A)$
we have\[
B(B_{1}h-A_{1}k)=B_{1}(Bh-Ak)=0,\]
 and hence $h\oplus k\in\mathcal{G}(B_{1}^{-1}A_{1})$ because $B$
is injective.
\begin{prop}
\label{pro:commutant-is-quotient-ring}Let $\mathcal{A}$ be a commutative
algebra with a rationally strictly cyclic vector $h_{0}$. For every
densely defined linear transformation $X$ commuting with $\mathcal{A}$,
such that $h_{0}\in\mathcal{D}(X)$, there exist $A,B\in\mathcal{A}$
such that $\ker B=\{0\}$ and $X\subset B^{-1}A$. If $X$ is bounded,
we have $X=B^{-1}A$. In particular, the commutant $\mathcal{A}'$
is a commutative algebra.\end{prop}
\begin{proof}
As in the preceding proof, we choose for each $h\in\mathcal{H}$ operators
$A_{h},B_{h}\in\mathcal{A}$ such that $\ker B_{h}=\{0\}$ and $B_{h}h=A_{h}h_{0}$.
Assume now that $h_{0}\in\mathcal{D}(X)$ and $X$ commutes with $\mathcal{A}$.
We have then for $h\in\mathcal{D}(X)$\begin{eqnarray*}
B_{h}B_{Xh_{0}}Xh & = & B_{Xh_{0}}XB_{h}h=B_{Xh_{0}}XA_{h}h_{0}\\
 & = & A_{h}B_{Xh_{0}}Xh_{0}=A_{h}A_{Xh_{0}}h_{0}\\
 & = & A_{Xh_{0}}B_{h}h=B_{h}A_{Xh_{0}}h,\end{eqnarray*}
from which we conclude that $X\subset B_{Xh_{0}}^{-1}A_{Xh_{0}}$
because $B_{h}$ is injective. The remaining assertions follow easily
from this.
\end{proof}
Sometimes an algebra with a rationally strictly cyclic vector has
the stronger property defined below.
\begin{defn}
\label{def:confluence}Let $\mathcal{A}\subset\mathcal{L}(\mathcal{H})$
be a unital algebra. We will say that $\mathcal{A}$ is \emph{confluent}
if for every two vectors $h_{1},h_{2}\in\mathcal{H}\setminus\{0\}$
there exist injective operators $B_{1},B_{2}\in\mathcal{A}$ such
that $B_{1}h_{1}=B_{2}h_{2}$.\end{defn}
\begin{prop}
\label{pro:confluent=00003Drsc+noKer}For a commutative unital algebra
$\mathcal{A}\subset\mathcal{L}(\mathcal{H})$, the following two assertions
are equivalent:
\begin{enumerate}
\item $\mathcal{A}$ has a rationally strictly cyclic vector and $\ker B=\{0\}$
for every $B\in\mathcal{A}\setminus\{0\}$;
\item $\mathcal{A}$ is confluent.
\end{enumerate}
If these equivalent conditions are satisfied, then every nonzero vector
is rationally strictly cyclic for $\mathcal{A}$; moreover, every
densely defined linear transformation commuting with $\mathcal{A}$
is contained in $B^{-1}A$ for some $A,B\in\mathcal{A}$ such that
$\ker B=\{0\}$.

\end{prop}
\begin{proof}
Assume first that (1) holds, and $h_{1},h_{2}\in\mathcal{H}\setminus\{0\}$.
With the notation used earlier, we have\[
A_{h_{2}}B_{h_{1}}h_{1}=A_{h_{2}}A_{h_{1}}h_{0}=A_{h_{1}}B_{h_{2}}h_{2}.\]
The operators $A_{h_{1}},A_{h_{2}}$ are not zero, and therefore $A_{h_{2}}B_{h_{1}},A_{h_{1}}B_{h_{2}}$
are injective by hypothesis.

Conversely, assume that $\mathcal{A}$ is confluent. Clearly, every
nonzero vector is then rationally strictly cyclic. It remains to show
that every $B\in\mathcal{A}\setminus\{0\}$ is injective. Assume to
the contrary that $Bh_{1}=0$ for some $h_{1}\ne0$, and choose $h_{2}\notin\ker B$.
If $B_{1},B_{2}$ are as in Definition \ref{def:confluence}, we obtain\[
0=B_{1}Bh_{1}=BB_{1}h_{1}=BB_{2}h_{2}=B_{2}Bh_{2}.\]
This implies $Bh_{2}=0$, contrary to the choice of $h_{2}$. The
last assertion follows from Proposition \ref{pro:commutant-is-quotient-ring}.
\end{proof}
As an application of the results in this section, we show that some
other algebras of Toeplitz operators have the closability property.
Consider a bounded, connected open set $\Omega\subset\mathbb{C}$
bounded by $n+1$ analytic simple Jordan curves, and fix a point $\omega_{0}\in\Omega$.
The algebra $H^{\infty}(\Omega)$ consists of the bounded analytic
functions on $\Omega$, while $H_{\omega_{0}}^{2}(\Omega)$ is defined
as the space of analytic functions $f$ on $\Omega$ with the property
that $|f|^{2}$ has a harmonic majorant in $\Omega$. The norm on
$H_{\omega_{0}}^{2}(\Omega)$ is defined as\[
\|f\|_{2}^{2}=\inf\{u(\omega_{0}):u\text{ a harmonic majorant of }|f|^{2}\},\quad f\in H_{\omega_{0}}^{2}(\Omega).\]
Multiplication by a function $u\in H^{\infty}(\Omega)$ determines
a bounded operator $T_{u}$ on $H_{\omega_{0}}^{2}(\Omega)$.
\begin{prop}
\label{pro:omega-toeplitz-algebra-is-cp}The constant function $1\in H_{\omega_{0}}^{2}(\Omega)$
is a rationally strictly cyclic vector for the algebra $\{T_{u}:u\in H^{\infty}(\Omega)\}$.
In particular, this algebra has the closability property.
\end{prop}
The statement is equivalent to the following result. We refer to \cite{A-D}
and \cite{fish} for the function theoretical background.
\begin{lem}
\label{lem:omega-f=00003Du/v}For every function $f\in H_{\omega_{0}}^{2}(\Omega)$
there exist $u,v\in H^{\infty}(\Omega)$ such that $v\not\equiv0$
and $vf=u$.\end{lem}
\begin{proof}
Denote by $\pi:\mathbb{D}\to\Omega$ a (universal) covering map such
that $\pi(0)=\omega_{0}$, and denote by $\Gamma$ the corresponding
group of deck transformations. Thus, $\Gamma$ consists of those analytic
automorphisms $\varphi$ of $\mathbb{D}$ with the property that $\pi\circ\varphi=\pi$.
The map $f\mapsto f\circ\pi$ is an isometry from $H_{\omega_{0}}^{2}(\Omega)$
onto the space of those functions $g\in H^{2}$ such that $g\circ\varphi=g$
for every $\varphi\in\Gamma$.

Fix now $f\in H_{\omega_{0}}^{2}(\Omega)$, and construct an outer
function $w\in H^{2}$ such that $|w(\zeta)|=\min\{1,1/|f\circ\pi(\zeta)|\}$
for almost every $\zeta\in\mathbb{T}$. The function $w$ is obviously
modulus automorphic in the sense that $|w\circ\varphi|=|w|$ for every
$\varphi\in\Gamma$. It follows that there is a group homomorphism
$\gamma:\Gamma\to\mathbb{T}$ such that $w\circ\varphi=\gamma(\varphi)w$
for every $\varphi\in\Gamma$. Choose a modulus automorphic Blaschke
product $b\in H^{\infty}$ such that $b\circ\varphi=\overline{\gamma(\varphi)}b$
for $\gamma\in\Gamma$; see \cite[Theorem 5.6.1]{fish} for the construction
of such products. Then there exist functions $u,v\in H^{\infty}(\Omega)$
such that $v\circ\pi=bw$ and $u\circ\pi=bw(f\circ\pi)$. These functions
satisfy the requirements of the lemma.
\end{proof}

\section{Quasisimilar Algebras\label{sec:Quasisimilar-Algebras}}

We will now study the effect of quasisimilarity on the closability
property and the existence of rationally strictly cyclic vectors.
Recall that an operator $Q\in\mathcal{L}(\mathcal{H}_{1},\mathcal{H}_{2})$
is called a \emph{quasiaffinity} if it is injective and has dense
range.
\begin{defn}
\label{def:quasiaffinity-of-algebras}An algebra $\mathcal{A}_{1}\subset\mathcal{L}(\mathcal{H}_{1})$
is a \emph{quasiaffine transform} of an algebra $\mathcal{A}_{2}\subset\mathcal{L}(\mathcal{H}_{2})$
if there exists a quasiaffinity $Q\in\mathcal{L}(\mathcal{H}_{1},\mathcal{H}_{2})$
such that, for every $T_{2}\in\mathcal{A}_{2}$ we have $QT_{1}=T_{2}Q$
for some $T_{1}\in\mathcal{A}_{1}$. We write $\mathcal{A}_{1}\prec\mathcal{A}_{2}$
if $\mathcal{A}_{1}$ is a quasiaffine transform of $\mathcal{A}_{2}$.
\end{defn}
The relation $\mathcal{A}_{1}\prec\mathcal{A}_{2}$ can simply be
written as $Q^{-1}\mathcal{A}_{2}Q\subset\mathcal{A}_{1}$ for some
quasiaffinity $Q$.
\begin{prop}
\label{propolema:quasiaffine-and-stuff} Assume that $\mathcal{A}_{1}\subset\mathcal{L}(\mathcal{H}_{1})$
and $\mathcal{A}_{2}\subset\mathcal{L}(\mathcal{H}_{2})$ are unital
algebras such that $\mathcal{A}_{1}\prec\mathcal{A}_{2}$.
\begin{enumerate}
\item If $\mathcal{A}_{1}$ is commutative, then $\mathcal{A}_{2}$ is commutative
as well.
\item If $\mathcal{A}_{2}$ has the closability property, then so does $\mathcal{A}_{1}$.
\item If $\mathcal{A}_{2}$ is confluent, then so is $\mathcal{A}_{1}$.
\end{enumerate}
\end{prop}
\begin{proof}
Let $Q$ be as in Definition \ref{def:quasiaffinity-of-algebras}.
Since the map $T\mapsto Q^{-1}TQ$ is obviously an injective algebra
homomorphism on $\mathcal{A}_{2}$, part (1) is immediate.

To prove (2), let $X$ be a densely defined linear transformation
commuting with $\mathcal{A}_{1}$. Define the linear transformation
$Y=QXQ^{-1}$ on the dense subspace $\mathcal{D}(Y)=Q\mathcal{D}(X)$.
Since all the operators $T_{2}\in\mathcal{A}_{2}$ have the property
that $Q^{-1}T_{2}Q$ is in $\mathcal{A}_{1}$, it follows easily that
$Y$ commutes with $\mathcal{A}_{2}$. Assume now that $\mathcal{A}_{2}$
has the closability property, so that $Y$ is closable. We will verify
that $X$ is closable as well. Assume that $h_{n}\in\mathcal{D}(X)$
are such that $h_{n}\to0$ and $Xh_{n}\to k$ as $n\to\infty$. Obviously
then $\mathcal{D}(Y)\ni Qh_{n}\to0$ and $YQh_{n}\to Qk$. We deduce
that $Qk=0$, and therefore $k=0$ since $Q$ is a quasiaffinity.

Finally, assume that $\mathcal{A}_{2}$ is confluent, and $h_{1},h_{2}\in\mathcal{H}\setminus\{0\}$.
We choose injective $C_{1},C_{2}\in\mathcal{A}_{2}$ so that $C_{1}Qh_{1}=C_{2}Qh_{2}$,
and observe that $B_{1}h_{1}=B_{2}h_{2}$, where $B_{j}=Q^{-1}C_{j}Q\in\mathcal{A}_{1}$
are injective.\end{proof}
\begin{defn}
\label{def:quasisimilarity-for-algs}An algebra $\mathcal{A}_{1}\subset\mathcal{L}(\mathcal{H}_{1})$
is \emph{quasisimilar} to an algebra $\mathcal{A}_{2}\subset\mathcal{L}(\mathcal{H}_{2})$
if there exist quasiaffinities $Q\in\mathcal{L}(\mathcal{H}_{1},\mathcal{H}_{2})$
and $R\in\mathcal{L}(\mathcal{H}_{2},\mathcal{H}_{1})$ such that
$Q^{-1}\mathcal{A}_{2}Q\subset\mathcal{A}_{1}$, $R^{-1}\mathcal{A}_{1}R\subset\mathcal{A}_{2}$,
$QR\in\mathcal{A}_{2}'$, and $RQ\in\mathcal{A}_{1}'$. We write $\mathcal{A}_{1}\sim\mathcal{A}_{2}$
if $\mathcal{A}_{1}$ is quasisimilar to $\mathcal{A}_{2}$.
\end{defn}
Using the proofs of parts (1) and (2) of the following result, it
is easy to see that quasisimilarity is an equivalence relation.
\begin{prop}
\label{pro:facts-about-quasisimilarity}Assume that $\mathcal{A}_{1}$
and $\mathcal{A}_{2}$ are commutative quasisimilar algebras, and
$Q,R$ satisfy the conditions of Definition \ref{def:quasisimilarity-for-algs}.
\begin{enumerate}
\item We have $Q^{-1}\mathcal{A}_{2}Q=\mathcal{A}_{1}$ and $R^{-1}\mathcal{A}_{1}R=\mathcal{A}_{2}$.
\item The maps $T_{2}\mapsto Q^{-1}T_{2}Q$ and $T_{1}\mapsto R^{-1}T_{1}R$
are mutually inverse algebra isomorphisms between $\mathcal{A}_{1}$
and $\mathcal{A}_{2}$.
\item The commutant $\mathcal{A}_{1}'$ is commutative if and only if $\mathcal{A}_{2}'$
is commutative.
\item If $h_{1}\in\mathcal{H}_{1}$ is rationally strictly cyclic for $\mathcal{A}_{1}$
then $Qh_{1}$ is rationally strictly cyclic for $\mathcal{A}_{2}$.
\item The algebra $\mathcal{A}_{1}$ is confluent if and only if $\mathcal{A}_{2}$
is confluent.
\item The algebra $\mathcal{A}_{1}'$ is confluent if and only if $\mathcal{A}_{2}'$
is confluent.
\item The algebra $\mathcal{A}_{1}'$ has the closability property if and
only if $\mathcal{A}_{2}'$ does.
\end{enumerate}
\end{prop}
\begin{proof}
Define $\Phi:\mathcal{A}_{2}\to\mathcal{A}_{1}$ and $\Psi:\mathcal{A}_{1}\to\mathcal{A}_{2}$
by setting $\Phi(T_{2})=Q^{-1}T_{2}Q$ and $\Psi(T_{1})=R^{-1}T_{1}R$.
We have\[
\Psi(\Phi(T_{2}))=R^{-1}Q^{-1}T_{2}QR=R^{-1}Q^{-1}QRT_{2}=T_{2},\quad T_{2}\in\mathcal{A}_{2},\]
and similarly $\Phi(\Psi(T_{1}))=T_{1}$ for $T_{1}\in\mathcal{A}_{1}$.
This proves (2), and (1) follows from (2).

Assume now that $\mathcal{A}_{1}'$ is commutative and $A,B\in\mathcal{A}_{2}'$.
We claim that $RAQ$ and $RBQ$ belong to $\mathcal{A}'_{1}$. Indeed,\begin{eqnarray*}
T_{1}RAQ & = & R(R^{-1}T_{1}R)AQ=RA(R^{-1}T_{1}R)Q\\
 & = & RAR^{-1}T_{1}(RQ)=RAR^{-1}(RQ)T_{1}=RAQT_{1}\end{eqnarray*}
for $T_{1}\in\mathcal{A}_{1}$. We deduce that $RAQRBQ=RBQRAQ$ and
hence $AQRB=BQRA$. Taking $A$ or $B$ to be the identity operator,
we deduce that $QR$ commutes with $B$ and $A$, and therefore $QRAB=QRBA$,
and finally the desired equality $AB=BA$.

To prove (4), assume that $h_{1}$ is rationally strictly cyclic for
$\mathcal{A}_{1}$. Proposition \ref{pro:commutant-is-quotient-ring}
implies the existence of $A_{1},B_{1}\in\mathcal{A}_{1}$ such that
$\ker B_{1}=\{0\}$ and $RQ=B_{1}^{-1}A_{1}$. Set $A_{2}=R^{-1}A_{1}R,B_{2}=R^{-1}B_{1}R\in\mathcal{A}_{2}$,
and observe that\[
B_{2}QR=R^{-1}(B_{1}RQ)R=R^{-1}A_{1}R=A_{2}.\]
To show that $Rh_{1}$ is rationally strictly cyclic for $\mathcal{A}_{2}$,
fix a vector $h_{2}\in\mathcal{H}_{2}$, and choose $S_{1},T_{1}\in\mathcal{A}_{1}$
such that $\ker T_{1}=\{0\}$ and $T_{1}Rh_{2}=S_{1}h_{1}$. Set now
$T_{2}=R^{-1}T_{1}R,S_{2}=R^{-1}S_{1}R\in\mathcal{A}_{2}$. We have\[
RQRT_{2}h_{2}=RQT_{1}Rh_{2}=RQS_{1}h_{1}=S_{1}RQh_{1}=RS_{2}Qh_{1},\]
so that $QRT_{2}h_{2}=S_{2}Qh_{1}$. Applying $B_{2}$ to both sides
we obtain $A_{2}T_{2}h_{2}=B_{2}S_{2}Qh_{1}$, and strict cyclicity
follows because $A_{2}T_{2},B_{2}S_{2}\in\mathcal{A}_{2}$ and $\ker(B_{2}S_{2})=\{0\}$.

Assertion (5) follows easily from (4) and Proposition \ref{pro:confluent=00003Drsc+noKer},
or directly from Proposition \ref{propolema:quasiaffine-and-stuff}(3).

Assume now that $\mathcal{A}'_{1}$ is confluent, and let $h,k\in\mathcal{H}_{2}$
be two nonzero vectors. Then there exist then injective operators
$A_{1},B_{1}\in\mathcal{A}'_{1}$ such that $A_{1}Rh=B_{1}Rk$. Thus
we have $A_{2}h=B_{2}k$, where $A_{2}=QA_{1}R$ and $B_{2}=QB_{1}R$
are injective operators in $\mathcal{A}'_{2}$. This proves (6).

Finally, assume that $\mathcal{A}_{1}'$ has the closability property,
and let $X$ be a densely defined linear transformation commuting
with $\mathcal{A}_{2}'$. As in the proof of Lemma \ref{propolema:quasiaffine-and-stuff}(2),
to prove (7) it will suffice to show that the linear transformation
$Y_{0}=QXQ^{-1}$ defined on the dense space $\mathcal{D}(Y_{0})=Q\mathcal{D}(X)$
is closable. To show this, we will define a linear transformation
$Y\supset Y_{0}$ which commutes with $\mathcal{A}'_{1}$. Its domain
$\mathcal{D}(Y)$ consists of all the finite sums of the form $\sum_{n}T_{n}Qh_{n}$,
where $T_{n}\in\mathcal{A}'_{1}$ and $h_{n}\in\mathcal{D}(X)$, and\[
Y\sum_{n}T_{n}Qh_{n}=\sum_{n}T_{n}QXh_{n}.\]
 To show that $Y$ is well-defined, it will suffice to prove that
$\sum_{n}T_{n}Qh_{n}=0$ implies $R\sum_{n}T_{n}QXh_{n}=0$. Indeed,
since $RT_{n}Q\in\mathcal{A}'_{2}$, we have $RT_{n}Qh_{n}\in\mathcal{D}(X)$
and \[
\sum_{n}RT_{n}QXh_{n}=\sum_{n}XRT_{n}Qh_{n}=XR\sum_{n}T_{n}Qh_{n}=0.\]
The fact that $Y$ commutes with every $T\in\mathcal{A}'_{1}$ is
easily verified. If $\sum_{n}T_{n}Qh_{n}\in\mathcal{D}(Y)$ then clearly
$\sum_{n}TT_{n}Qh_{n}\in\mathcal{D}(Y)$, and\[
YT\sum_{n}T_{n}Qh=\sum_{n}TT_{n}QXh_{n}=TY\sum_{n}T_{n}Qh_{n}.\]
The inclusion $Y\supset Y_{0}$ is verified by taking $T_{n}=I$.
\end{proof}
We will be using the results in this section for the special case
of algebras generated by a completely nonunitary contraction $T\in\mathcal{L}(\mathcal{H})$.
For such a contraction we will write\[
H^{\infty}(T)=\{u(T):u\in H^{\infty}\}.\]

Parts (1) and (2) of the following lemma are easily verified; in fact
Definition \ref{def:quasiaffinity-of-algebras} was formulated so
as to make part (2) correct.
\begin{lem}
\label{lem:quasisim-for-cnu-and-for-hinfty}Let $T_{1}$ and $T_{2}$
be two completely nonunitary contractions.
\begin{enumerate}
\item If $T_{1}\prec T_{2}$ then $H^{\infty}(T_{1})\prec H^{\infty}(T_{2}).$
\item If $T_{1}\sim T_{2}$ then $H^{\infty}(T_{1})\sim H^{\infty}(T_{2})$.
\item If $H^{\infty}(T_{1})\sim H^{\infty}(T_{2})$ and $T_{1}$ is of class
$C_{0}$, then $T_{2}$ is also of class $C_{0}$.
\item If $H^{\infty}(T_{1})\sim H^{\infty}(T_{2})$ and $T_{1}$ is not
of class $C_{0}$, then $T_{1}\sim\varphi(T_{2})$ for some conformal
automorphism $\varphi$ of $\mathbb{D}$.
\end{enumerate}
\end{lem}
\begin{proof}
To prove (3), observe that $H^{\infty}(T_{1})\sim H^{\infty}(T_{2})$
implies that $H^{\infty}(T_{2})$ is isomorphic to $H^{\infty}(T_{1})$.
Assume that $T_{2}$ is of class $C_{0}$. If $T_{1}$ is a scalar
multiple of the identity, then $H^{\infty}(T_{1})=\mathbb{C}I$, and
therefore $H^{\infty}(T_{2})=\mathbb{C}I$ and then $T_{2}$ must
be a scalar multiple of the identity, hence of class $C_{0}$. If
$T_{1}$ is not a scalar multiple of the identity, then $H^{\infty}(T_{1})$
has zero divisors. Indeed, in this case the minimal function $m$
of $T_{1}$ can be factored into a product $m=m_{1}m_{2}$ of two
nonconstant inner functions, and then $m_{1}(T_{1})\ne0\ne m_{2}(T_{1})$
while $m_{1}(T_{1})m_{2}(T_{1})=0$. We conclude that $H^{\infty}(T_{2})$
must also have zero divisors, and this obviously implies that $T_{2}$
is of class $C_{0}$.

Finally, assume that $H^{\infty}(T_{1})\sim H^{\infty}(T_{2})$ and
$T_{1}$ (as well as $T_{2}$ by part (3)) is not of class $C_{0}$.
Let $Q$ and $R$ be quasiaffinities satisfying the conditions of
Definition \ref{def:quasiaffinity-of-algebras} for the algebras $\mathcal{A}_{1}=H^{\infty}(T_{1})$
and $\mathcal{A}_{2}=H^{\infty}(T_{2})$. The hypothesis implies that
the maps $u\mapsto u(T_{1})$ and $u\mapsto u(T_{2})$ are algebra
isomorphisms from $H^{\infty}$ to $H^{\infty}(T_{1})$ and $H^{\infty}(T_{2})$,
respectively. Thus, for every $u\in H^{\infty}$ there exists a unique
$v\in H^{\infty}$ satisfying $v(T_{2})=R^{-1}u(T_{1})R$. The map
$\Phi:u\mapsto v$ is an algebra automorphism of $H^{\infty}$. In
particular, the function $\varphi=\Phi(\text{id}_{\mathbb{D}})$ must
have spectrum (in $H^{\infty}$) equal to $\overline{\mathbb{D}}$,
so that $\varphi(\mathbb{D})=\mathbb{D}$. We claim that $\Phi(u)=u\circ\varphi$
for every $u\in H^{\infty}$. Indeed, given $\lambda\in\mathbb{D}$,
we can factor $u(z)-u(\varphi(\lambda))=(z-\lambda)w$ for some $w\in H^{\infty}$,
so that $\Phi(u)-u(\varphi(\lambda))=(\varphi-\varphi(\lambda))\Phi(w)$.
The equality $(\Phi(u))(\lambda)=u(\varphi(\lambda))$ follows immediately.
Since $\Phi$ is an automorphism, it follows that $\varphi$ is a
conformal automorphism of $\mathbb{D}$, and clearly $T_{1}\sim\varphi(T_{2})$.\end{proof}
\begin{cor}
\label{cor:Tqsim-S-implies-confl}Let $T$ be a completely nonunitary
contraction. If $T\sim S$ then $H^{\infty}(T)$ is confluent. If
$T\sim S(m)$ then $H^{\infty}(T)$ has a rationally strictly cyclic
vector.\end{cor}
\begin{proof}
It suffices to observe that $H^{\infty}(S)=\mathcal{W}_{S}$, $H^{\infty}(S(m))=\mathcal{W}_{S(m)}$,
and to apply Proposition \ref{pro:facts-about-quasisimilarity}(5)
and (4).
\end{proof}
For operators of class $C_{0}$, the converse of the preceding result
is also true. The case of confluent algebras of the form $H^{\infty}(T)$
will be discussed more thoroughly in the remaining two sections of
the paper.
\begin{prop}
\label{pro:strictly-cyclic-and-C0}Assume that $T$ is a completely
nonunitary contraction such that $H^{\infty}(T)$ has a rationally
strictly cyclic vector.
\begin{enumerate}
\item If there exists $f\in H^{\infty}\setminus\{0\}$ such that $\ker f(T)\ne\{0\}$,
 then $T$ is of class $C_{0}$ and $T\sim S(m)$, where $m$ is the
minimal function of $T$.
\item If $\ker f(T)=\{0\}$ for every $f\in H^{\infty}\setminus\{0\}$,
then $H^{\infty}(T)$ is confluent.
\end{enumerate}
\end{prop}
\begin{proof}
Part (2) follows immediately from Proposition \ref{pro:confluent=00003Drsc+noKer}.
To verify (1), assume that $f\in H^{\infty}\setminus\{0\}$, $\ker f(T)\ne\{0\}$,
and $H^{\infty}(T)$ has a rationally strictly cyclic vector $h_{0}\in\mathcal{H}$.
Choose a nonzero vector $h_{1}\in\ker f(T)$, and functions $u_{1},v_{1}\in H^{\infty}$
such that $v_{1}(T)$ is injective and $v_{1}(T)h_{1}=u_{1}(T)h_{0}$.
The function $u_{1}$ is not zero since $v_{1}(T)h_{1}\ne0$. We claim
that $f(T)u_{1}(T)=0$. Indeed, let $h$ be an arbitrary vector in
$\mathcal{H}$. Choose $u,v\in H^{\infty}$ such that $v(T)$ is injective
and $v(T)h=u(T)h_{0}$. We have then\begin{eqnarray*}
v(T)[f(T)u_{1}(T)h] & = & f(T)u_{1}(T)[v(T)h]=f(T)u_{1}(T)u(T)h_{0}\\
 & = & f(T)u(T)u_{1}(T)h_{0}=f(T)u(T)v_{1}(T)h_{1}=0,\end{eqnarray*}
and therefore $f(T)u_{1}(T)h=0$. Thus $T$ is of class $C_{0}$ because
$(fu_{1})(T)=0$ and $fu_{1}\in H^{\infty}\setminus\{0\}$.

Finally, let $m$ be the minimal function of $T$, denote by $\mathcal{M}$
the cyclic space for $T$ generated by $h_{0}$, and set $\mathcal{N}=\mathcal{M}^{\perp}$.
Let $T'=P_{\mathcal{N}}T|\mathcal{N}$ be the compression of $T$
to $\mathcal{N}$. We have proved $m(T')=0$. Let now $h\in\mathcal{H}$
be a vector, and pick $u,v\in H^{\infty}$ such that $v(T)$ is injective
and $v(T)h=u(T)h_{0}$. In particular, we have $v(T')h=0$. The injectivity
of $v(T)$ is equivalent to the condition $v\wedge m=1$, and this
implies that $v(T')$ is injective as well, so that $h=0$ We proved
therefore that $\mathcal{M}=\mathcal{H}$. In other words, $T$ has
a cyclic vector, and thus $T\sim S(m)$ by the results of \cite{complements2}
(see also \cite[Theorem III.2.3]{Berco}).
\end{proof}
We conclude this section with a result about arbitrary operators of
class $C_{0}$.
\begin{prop}
\label{pro:C0-commutant-is-CP}For any operator $T$ of class $C_{0}$,
the commutant $\{T\}'$ has the closability property.\end{prop}
\begin{proof}
The operator $T$ is quasisimilar to an operator of the form $T'=\bigoplus_{i\in I}S(m_{i})$,
where each $m_{i}$ is an inner function; see \cite[Theorem III.5.1]{Berco}.
By Proposition \ref{pro:facts-about-quasisimilarity}(7), it suffices
to show that $\{T'\}'$ has the closability property. Now, $\{T'\}'\supset\bigoplus_{i\in I}\{S(m_{i})\}'$,
and Lemma \ref{lem:CP-for-direct-sums} shows that it suffices to
show that $\{S(m)\}'$ has the closability property for each inner
function $m$. This follows from Proposition \ref{pro:SandS(m)haveCP}
because $\{S(m)\}'=\mathcal{W}_{S(m)}$.
\end{proof}

\section{Confluent Algebras of the Form $H^{\infty}(T)$\label{sec:Confluent-Algebras-of}}

Consider a completely nonunitary contraction $T\in\mathcal{L}(\mathcal{H})$
such that $H^{\infty}(T)$ has a rationally strictly cyclic vector.
According to Proposition \ref{pro:strictly-cyclic-and-C0}, we have
$T\sim S(m)$ if any nonzero operator in $H^{\infty}(T)$ has nonzero
kernel. Therefore we will restrict ourselves now to operators $T$
such that $f(T)$ is injective for every nonzero element of $H^{\infty}$.
In other words, we will assume that $H^{\infty}(T)$ is a confluent
algebra (cf. Proposition \ref{pro:confluent=00003Drsc+noKer}) and
$\dim\mathcal{H}>1$. In this case, the space $\mathcal{H}$ can be
identified with a space of meromorphic functions. Let us denote by
$N$ the Nevanlinna class consisting of those meromorphic functions
in $\mathbb{D}$ which can be written as $u/v$, with $u,v\in H^{\infty}$.
\begin{lem}
\label{lem:u/v=00003Dh/h0}Assume that $T$ is a completely nonunitary
contraction such that $H^{\infty}(T)$ is confluent. Let $h,h_{0}$
be two vectors such that $h_{0}\ne0$, and choose $u,v\in H^{\infty}$,
$v\ne0$, such that $v(T)h=u(T)h_{0}$. The function $u/v\in\mathcal{M}(\mathbb{D})$
is uniquely determined by $h$ and $h_{0}$. We have $u/v=0$ if and
only if $h=0$.\end{lem}
\begin{proof}
Choose another pair of functions $u_{1},v_{1}\in H^{\infty}$, $v_{1}\ne0$,
satisfying $v_{1}(T)h=u_{1}(T)h_{0}$. We have \[
(v_{1}(T)u(T)-v(T)u_{1}(T))h_{0}=(v_{1}(T)v(T)-v(T)v_{1}(T))h=0,\]
and therefore $h_{0}\in\ker(v_{1}u-uv_{1})(T)$. The hypothesis implies
that $v_{1}u=vu_{1}$ and hence $u/v=u_{1}/v_{1}$.
\end{proof}
The function $u/v$ will be denoted $h/h_{0}$. It is clear that the
map $h\mapsto h/h_{0}$ is an injective linear map from $\mathcal{H}$
to $N$, and $u(T)h/u(T)h_{0}=h/h_{0}$ if $u\in H^{\infty}\setminus\{0\}$.
We also have\[
\frac{h}{h_{0}}=\frac{h}{h_{1}}\cdot\frac{h_{1}}{h_{0}}\]
provided that $h_{0},h_{1}\in\mathcal{H}\setminus\{0\}$. Now let
$h,h_{0}\in\mathcal{H}\setminus\{0\}$. There exists a unique integer
$n$ such that the nonzero function $h/h_{0}$ can be written as\[
\frac{h}{h_{0}}(z)=z^{n}\frac{u(z)}{v(z)}\]
with $u,v\in H^{\infty}$ and $u(0)\ne0\ne v(0)$. The number $n$
will be denoted ${\rm ord}_{0}(h/h_{0})$. It will be convenient to
write ${\rm ord}_{0}(h/h_{0})=\infty$ if $h=0$.
\begin{lem}
\label{pro:H=00003DD_n}Let $T$ be a completely nonunitary contraction
such that $H^{\infty}(T)$ is confluent. Then $0\ge\inf\{{\rm ord}_{0}(h/h_{0}):h\in\mathcal{H}\}>-\infty$
for every $h_{0}\in\mathcal{H}\setminus\{0\}$.\end{lem}
\begin{proof}
Clearly ${\rm ord}_{0}(h_{0}/h_{0})=0$. The sets\[
\mathcal{H}_{n}=\{h\in\mathcal{H}:{\rm ord}_{0}(h/h_{0})\ge-n\},\quad n=0,1,\dots,\]
are linear manifolds such that $\bigcup_{n\ge0}\mathcal{H}_{n}=\mathcal{H}$.
Let now $\mathcal{D}_{n,k}$, $k\ge1$, denote the set of all vectors
$h\in\mathcal{H}_{n}$ for which $h/h_{0}$ can be written as \[
\frac{h}{h_{0}}(z)=z^{-n}\frac{u(z)}{v(z)}\]
 with $\|u\|_{\infty},\|v\|_{\infty}\le k$ and $|u(0)|,|v(0)|\ge1$.
Observe that \[
\bigcup_{m\le n,k\ge1}\mathcal{D}_{m,k}=\mathcal{H}_{n}\setminus\{0\}.\]
The proposition will follow if we can show that one of the sets $\mathcal{D}_{n,k}$
has an interior point, and this will follow from the Baire category
theorem once we  prove that each $\mathcal{D}_{n,k}$ is closed. Assume
indeed that $h_{i}\in\mathcal{D}_{n,k}$ is a sequence such that $h_{i}\to h$
as $i\to\infty$. For each $i$ write\[
\frac{h_{i}}{h_{0}}(z)=z^{-n}\frac{u_{i}}{v_{i}}\]
 with $\|u_{i}\|_{\infty},\|v_{i}\|_{\infty}\le k$ and $|u_{i}(0)|,|v_{i}(0)|\ge1$.
By the Vitali-Montel theorem we can assume, after dropping to a subsequence,
that there exist functions $u,v\in H^{\infty}$ such that $u_{i}(z)\to u(z)$
and $v_{i}(z)\to v(z)$ uniformly for $z$ in a compact subset of
$\mathbb{D}$. Clearly $\|u\|_{\infty},\|v\|_{\infty}\le k$ and $|u(0)|,|v(0)|\ge1$.
Moreover, we have $u_{i}(T)h_{0}\to u(T)h_{0}$ and $v_{i}(T)h_{i}\to v(T)h$
in the weak topology. (For the second sequence we need to write\[
v_{i}(T)h_{i}-v(T)h=v_{i}(T)(h_{i}-h)+(v_{i}(T)-v(T))h,\]
and use the fact that the first term tends to zero in norm, while
the second tends to zero weakly by \cite[Lemmas II.1.6 and II.1.7]{Sz-N-F}.)
The identities $T^{n}v_{i}(T)h_{i}=u_{i}(T)h_{0}$ therefore imply
$T^{n}v(T)h=u(T)h_{0}$ so that $h/h_{0}=z^{-n}u/v$, and thus $h\in\mathcal{D}_{n,k}$,
as desired.\end{proof}
\begin{lem}
\label{lem:0-in-spectrum-ofT}Let $T$ be a completely nonunitary
contraction such that $H^{\infty}(T)$ is confluent. Then $T$ is
injective and $T\mathcal{H}$ is a closed subspace of codimension
$1$. Thus $T$ is a Fredholm operator with ${\rm index}(T)=-1$.\end{lem}
\begin{proof}
The operator $T$ belongs to a confluent algebra, hence it is injective.
Note next that \[
{\rm ord}_{0}(Th/h_{0})={\rm ord}_{0}(h/h_{0})+1\]
and hence\[
\inf\{{\rm ord}_{0}(h/h_{0}):h\in\mathcal{H}\}+1=\inf\{{\rm ord}_{0}(h/h_{0}):h\in T\mathcal{H}\}.\]
Since these numbers are finite, we cannot have $T\mathcal{H}=\mathcal{H}$.
To conclude the proof, it will suffice to show that $T\mathcal{H}$
has codimension one since this implies that it is closed as well.
Choose $h_{0}\in\mathcal{H}\setminus T\mathcal{H}$, and note that
$\mbox{{\rm ord}}_{0}(h/h_{0})\ge0$ for every $h$. Indeed, ${\rm ord}_{0}(h/h_{0})=-n<0$
implies an identity of the form \[
T^{n}v(T)h=u(T)h_{0}\]
with $u(0)\ne0$.Factoring $u(z)-u(0)=zw(z)$, we obtain\[
h_{0}=\frac{1}{u(0)}T(T^{n-1}v(T)h-w(T)h_{0})\in T\mathcal{H},\]
a contradiction. Thus the function $h/h_{0}$ is analytic at 0, and
we can therefore define a linear functional $\Phi:\mathcal{H}\to\mathbb{C}$
by setting $\Phi h=(h/h_{0})(0)$. We will show that $\ker\Phi\subset T\mathcal{H}$.
Indeed, $h\in\ker\Phi$ implies that $v(T)h=Tu(T)h_{0}$ for some
$u,v\in H^{\infty}$ with $v(0)\ne0$. Factoring again $v(z)-v(0)=zw(z),$we
obtain\[
h=\frac{1}{v(0)}T(u(T)h_{0}-w(T)h)\in T\mathcal{H},\]
as claimed. Thus $T\mathcal{H}$ has codimension 1, and the lemma
is proved.
\end{proof}
The preceding results allow us to describe completely the spectral
picture of $T$, as well as its commutant. The argument for (3) already
appears in \cite{C-D}, and is included for the reader's convenience.
\begin{thm}
\label{thm:spectral-picture-for-confluence}Let $T\in\mathcal{L}(\mathcal{H})$
be a completely nonunitary contraction such that $H^{\infty}(T)$
is confluent.
\begin{enumerate}
\item We have $\sigma(T)=\overline{\mathbb{D}}$ and $\sigma_{{\rm e}}(T)=\mathbb{T}.$
\item For each $\lambda\in\mathbb{D}$, $\lambda I-T$ is injective and
has closed range of codimension $1$.
\item $\bigvee\{\ker(\lambda I-T^{*}):\lambda\in\mathbb{D}\}=\mathcal{H}.$
More generally, $\bigvee\{\ker(\lambda I-T^{*}):\lambda\in S\}=\mathcal{H}$
whenever the set $S\subset\mathbb{D}$ has an accumulation point in
$\mathbb{D}$.
\item For every nonzero invariant subspace $\mathcal{M}$ of $T$, there
exists an inner function $m\in H^{\infty}$ such that $\overline{m(T)\mathcal{H}}=\mathcal{M}$
and the compression $T_{\mathcal{M}^{\perp}}$ is quasisimilar to
$S(m)$. Conversely, for every inner function $m$, the minimal function
of $T_{(m(T)\mathcal{H})^{\perp}}$ is $m$.
\item $\{T\}'=H^{\infty}(T)$.
\item The operator $T$ is of class $C_{10}$. Thus, the powers $T^{*n}$
converge strongly to zero and $\lim_{n\to\infty}\|T^{n}h\|\ne0$ for
$h\in\mathcal{H}\setminus\{0\}$.
\end{enumerate}
In particular, properties $(2)$ and $(3)$ say that $T^{*}$ belongs
to the class $\mathcal{B}_{1}(\mathbb{D})$ defined in \cite{C-D}.

\end{thm}
\begin{proof}
For $\lambda\in\mathbb{D}$, the operator $T_{\lambda}=(I-\overline{\lambda}T)^{-1}(T-\lambda I)$
is also a completely nununitary contraction, and $H^{\infty}(T_{\lambda})=H^{\infty}(T)$
is confluent. Thus Lemma \ref{lem:0-in-spectrum-ofT} implies immediately
(2). In turn, (1) follows from (2) since $T$ is a contraction.

Next we prove (4). Let $\mathcal{M}\ne\{0\}$ be invariant for $T$,
set $\mathcal{N}=\mathcal{M}^{\perp}$, and choose $h_{0}\in\mathcal{M}\setminus\{0\}$.
Denote by $T'=P_{\mathcal{N}}T|\mathcal{N}$ the compression of $T$
to $\mathcal{M}$. Given $h\in\mathcal{N},$ an equality of the form
$v(T)h=u(T)h_{0}$ implies $v(T)h\in\mathcal{M}$, and therefore $v(T')h=0$.
The fact that $h_{0}$ is rationally strictly cyclic implies that
$T'$ is locally of class $C_{0}$, and hence of class $C_{0}$ by
\cite{local-C_0} (see also \cite[Theorem II.3.6]{Berco}). Denote
by $m$  the minimal function of $T'$. We show next that $T'$ has
a cyclic vector, hence it is quasisimilar to $S(m)$. Assume to the
contrary that $T'$ does not have a cyclic vector, and let $\mathcal{N}_{1},\mathcal{N}_{2}$
be cyclic spaces for $T'$ generated by two nonzero vectors $h_{1},h_{2}$
such that $T'|\mathcal{N}_{1}\sim S(m)$ and $\mathcal{N}_{1}\cap\mathcal{N}_{2}=\{0\}$
(see \cite{complements2} or \cite[Theorem III.2.13]{Berco}). There
exist nonzero functions $u_{1},u_{2}\in H^{\infty}$ such that $u_{1}(T)h_{1}=u_{2}(T)h_{2}$.
Dividing these functions by their greatest common inner divisor, we
may assume that $u_{1}$ and $u_{2}$ do not have any common inner
factor. We also have $u_{1}(T')h_{1}=u_{2}(T')h_{2}\in\mathcal{N}_{1}\cap\mathcal{N}_{2}$,
hence these vectors are equal to zero. We deduce that $m$ divides
$u_{1}$, and hence $m\wedge u_{2}=1$. This last equality implies
that $u_{2}(T')$ is a quasiaffinity, hence $u_{2}(T')h_{2}\ne0$,
a contradiction. Thus $T'$ is indeed cyclic. Observe now that have
$m(T)\mathcal{H}=m(T)\mathcal{M}+m(T')\mathcal{N}\subset\mathcal{M}$.
Denote now $\mathcal{M}_{1}=\overline{m(T)\mathcal{H}}$, $\mathcal{N}_{1}=\mathcal{M}_{1}^{\perp}$,
and $T_{1}=P_{\mathcal{N}_{1}}T|\mathcal{N}_{1}$. Clearly $m(T_{1})=0$,
and $T^{\prime*}=T_{1}^{*}|\mathcal{N}$. It follows that the minimal
function of $T_{1}$ is also $m$. Since $T_{1}$ has a cyclic vector,
it follows that $\mathcal{M}=\mathcal{M}_{1}$ by the results of \cite{complements2}
(see also \cite[Theorem III.2.13]{Berco}).

We start next with a a given inner function $m$, and denote by $m_{1}$
the minimal function of $T_{(m(T)\mathcal{H})^{\perp}}$. The function
$m_{1}$ must divide $m$, so that $m=m_{1}m_{2}$ for some other
inner function $m_{2}$. With the notation $\mathcal{H}_{1}=\overline{m_{1}(T)\mathcal{H}}=\overline{m(T)\mathcal{H}}$,
$T_{1}=T|\mathcal{H}_{1}$, the algebra $H^{\infty}(T_{1})$ is confluent,
and\[
\overline{m_{2}(T_{1})\mathcal{H}_{1}}=\overline{m_{2}(T)m_{1}(T)\mathcal{H}}=\overline{m(T)\mathcal{H}}=\mathcal{H}_{1},\]
so that $m_{2}(T_{1})$ has dense range. We claim that $\overline{m_{2}(T_{1})\mathcal{M}}=\mathcal{M}$
for every invariant subspace $\mathcal{M}$ for $T_{1}$. Indeed,
from the first part of (4) we know that $\mathcal{M}=\overline{m_{3}(T_{1})\mathcal{H}_{1}}$
for some inner function $m_{3}$. Hence\[
\overline{m_{2}(T_{1})\mathcal{M}}=\overline{m_{2}(T_{1})m_{3}(T_{1})\mathcal{H}_{1}}=\overline{m_{3}(T_{1})\overline{m_{2}(T_{1})\mathcal{H}_{1}}}=\overline{m_{3}(T_{1})\mathcal{H}_{1}}=\mathcal{M},\]
as claimed. Since $H^{\infty}(T_{1})$ is  confluent, we have $\sigma(T_{1})=\overline{\mathbb{D}}$
by part (1) of the theorem. This implies that $T_{1}$ belongs to
the class $\mathbb{A}$ defined in \cite{B-F-P}. By the results of
\cite{bcp}, there exist vectors $x,y\in\mathcal{H}_{1}$ such that\[
\langle u(T)x,y\rangle=\frac{1}{2\pi}\int_{0}^{2\pi}(1-m_{2}(0)\overline{m_{2}(e^{it})})u(e^{it})\, dt\]
for all $u\in H^{\infty}$. In particular, $\langle v(T)m_{2}(T)x,y\rangle=0$
for $v\in H^{\infty}$. Set $\mathcal{M}=\bigvee\{T^{n}x:n\ge0\}$,
and observe now that $y\perp m_{2}(T)\mathcal{M}$, and therefore
$y\perp\mathcal{M}$ as well. In particular,\[
0=\langle x,y\rangle=\frac{1}{2\pi}\int_{0}^{2\pi}(1-m_{2}(0)\overline{m_{2}(e^{it})})\, dt=1-|m_{2}(0)|^{2},\]
and this implies that $m_{2}$ is a constant function. We reach the
desired conclusion that the minimal function of $T_{(m(T)\mathcal{H})^{\perp}}$
is $m$.

To prove (3), assume that $S\subset\mathbb{D}$ has an accumulation
point in $\mathbb{D}$, and note that the space $\mathcal{N}=\bigvee\{\ker(\lambda I-T^{*}):\lambda\in S\}$
is invariant for $T^{*}$, and therefore $\mathcal{M}=\mathcal{N}^{\perp}$
is invariant for $T$. If $\mathcal{M}\ne\{0\}$, we have then $m(T)\mathcal{H}\subset\mathcal{M}$
for some inner function $m$, and therefore $\ker m(T)^{*}\supset\mathcal{N}$.
Given $\lambda\in S$, choose a nonzero vector $f_{\lambda}\in\ker(\lambda I-T)^{*}$,
and observe that $0=m(T)^{*}f_{\lambda}=\overline{m(\lambda)}f_{\lambda}$.
Thus $m(\lambda)=0$ for $\lambda\in S$, and we conclude that $m=0$,
which is impossible. This contradiction implies that $\mathcal{M}=\{0\}$,
thus verifying (3).

Consider next an operator $X\in\{T\}'=H^{\infty}(T)'$. By Proposition
\ref{pro:commutant-is-quotient-ring}, there exist $u,v\in H^{\infty}$,
$v\ne0$, so that $v(T)X=u(T)$. With $f_{\lambda}$ as above, we
have \[
\overline{v(\lambda)}X^{*}f_{\lambda}=(v(T)X)^{*}f_{\lambda}=u(T)^{*}f_{\lambda}=\overline{u(\lambda)}f(\lambda),\]
and thus\[
\left|\frac{u(\lambda)}{v(\lambda)}\right|=\frac{\|X^{*}f_{\lambda}\|}{\|f_{\lambda}\|}\le\|X^{*}\|.\]
 We deduce that $w=u/v\in H^{\infty}$ and $X=w(T)$.

The fact that the powers of $T^{*}$ tend strongly to zero follows
from (3) because $T^{*n}f_{\lambda}=\overline{\lambda}^{n}f_{\lambda}\to0$
as $n\to\infty$ for $\lambda\in\mathbb{D}$. It remains to prove
that the space\[
\mathcal{M}=\{h\in\mathcal{H}:\lim_{n\to\infty}\|T^{n}h\|=0\}\]
is equal to $\{0\}$. Assume to the contrary that $\mathcal{M}\ne\{0\}$,
and observe that $H^{\infty}(T|\mathcal{M})$ is also confluent. In
particular, $\sigma(T|\mathcal{M})=\overline{\mathbb{D}}$ and $T|\mathcal{M}$
is of class $C_{00}$. According to \cite{brown} and \cite[Theorem 6.6]{B-F-P},
$T|\mathcal{M}$ belongs to the class $\mathbb{A}_{\aleph_{0}}$,
and by \cite[Corollary 5.5]{B-F-P} $T$ has a further invariant subspace
$\mathcal{N}\subset\mathcal{M}$ such that $\mathcal{N}\ominus\overline{T\mathcal{N}}$
has infinite dimension. This space must however have dimension $1$
because $H^{\infty}(T|\mathcal{N})$ is confluent. This contradiction
shows that we must have $\mathcal{M}=\{0\}$, as claimed.
\end{proof}
Recall that $N_{+}\subset N$ denotes the collection of functions
of the form $u/v$, where $u,v\in H^{\infty}$ and $v$ is outer.
\begin{cor}
\label{cor:v-always-outer}Let $T\in\mathcal{L}(\mathcal{H})$ be
a completely nonunitary contraction such that $H^{\infty}(T)$ is
confluent, and fix a vector $h_{0}\in\ker T^{*}$. Assume that $\mathcal{H}=\bigvee\{T^{n}h_{0}:n\ge0\}$,
that is $h_{0}$ is cyclic for $T$. Then $h/h_{0}\in N_{+}$ for
every $h\in\mathcal{H}$.\end{cor}
\begin{proof}
Choose functions $u,u_{0}\in H^{\infty}$ such that $u_{0}/u=h/h_{0}$.
Thus we have $u_{0}(T)h_{0}=u(T)h$. Consider the factorizations $u=mv$
and $u_{0}=m_{0}v_{0}$, where $m,m_{0}$ are inner and $v,v_{0}$
are outer. By \cite[Proposition III.3.1]{Sz-N-F}, the operator $v_{0}(T)$
is a quasiaffinity, and therefore \[
\overline{m_{0}(T)\mathcal{H}_{0}}=\bigvee_{n\ge0}T^{n}v_{0}(T)m_{0}(T)h_{0}=\bigvee_{n\ge0}T^{n}v(T)m(T)h\subset\overline{m(T)\mathcal{H}}.\]
It follows that $(m(T)\mathcal{H})^{\perp}\subset(m_{0}(T)\mathcal{H})^{\perp}$
and thus $m$ divides $m_{0}$ by Theorem \ref{thm:spectral-picture-for-confluence}(4).
It follows that \[
\frac{h}{h_{0}}=\frac{u_{0}}{u}=\frac{v_{0}(m_{0}/m)}{v}\in N_{+},\]
as claimed.
\end{proof}
We will denote by $A$ the \emph{disk algebra}. This consists of those
functions in $H^{\infty}$ which are restrictions of continuous functions
on $\overline{\mathbb{D}}$. If $T$ is a completely nonunitary contraction,
we set $A(T)=\{u(T):u\in A\}$.
\begin{cor}
\label{cor:A(T)-never-confluent}Consider an operator $T\in\mathcal{L}(\mathcal{H})$,
where $\mathcal{H}$ is an infinite dimensional Hilbert space.
\begin{enumerate}
\item The algebra $\mathcal{P}_{T}$ is not confluent.
\item If $T$ is a completely nonunitary contraction, then $A(T)$ is not
confluent.
\end{enumerate}
\end{cor}
\begin{proof}
In proving (1), there is no loss of genrality in assuming that $\|T\|<1$
since $\mathcal{P}_{T}=\mathcal{P}_{\alpha T}$ for any $\alpha>0$.
Under this assumption, we have $\mathcal{P}_{T}\subset A(T)$, so
that it suffices to prove part (2). Assume therefore that $T$ is
a completely nonunitary contraction and $A(T)$ is confluent. The
larger algebra $H^{\infty}(T)$ is confluent as well, and \ref{pro:commutant-is-quotient-ring}
implies that for every $f\in H^{\infty}$, the operator $f(T)\in\{T\}'$
can be written as $f(T)=v(T)^{-1}u(T)$ with $u,v\in A$, $v\ne0$.
We have then $v(T)f(T)=u(T)$, and thus $f=u/v$. It is known however
that there are functions in $H^{\infty}$ which cannot be represented
as quotients of elements of $A$. An example is provided by any singular
inner function \[
f(\lambda)=e^{-\int_{\mathbb{T}}\frac{\zeta+\lambda}{\zeta-\lambda}\, d\mu(\zeta)},\quad\lambda\in\mathbb{D},\]
such that the closed support of the singular measure $\mu$ is the
entire circle $\mathbb{T}$.
\end{proof}
The assertion in Proposition \ref{pro:commutant-is-quotient-ring},
concerning unbounded linear transformations, can be improved when
$H^{\infty}(T)$ is confluent.
\begin{prop}
\label{pro:confluence-and-unbounded-commutant}Let $T\in\mathcal{L}(\mathcal{H})$
be a completely nonunitary contraction such that $H^{\infty}(T)$
is confluent. Every closed, densely defined linear transformation
commuting with $T$ is of the form $v(T)^{-1}u(T)$, where $u,v\in H^{\infty}$
and $v$ is an outer function.\end{prop}
\begin{proof}
Let $X$ be a closed, densely defined linear transformation commuting
with $T$. Since $X$ is closed, it must also commute with every operator
in $H^{\infty}(T)$. By Proposition \ref{pro:commutant-is-quotient-ring},
there exist $u,v\in H^{\infty}$ such that $v\not\equiv0$ and $X\subset v(T)^{-1}u(T)$.
Let us set \[
T_{1}=(T\oplus T)|\mathcal{G}(v(T)^{-1}u(T)),\]
and observe that the quasiaffinity $Q:h\oplus k\mapsto h$ from $\mathcal{G}(v(T)^{-1}u(T))$
to $\mathcal{H}$ satisfies $QT_{1}=TQ$. Thus $H^{\infty}(T_{1})\prec H^{\infty}(T)$,
and therefore $H^{\infty}(T_{1})$ is confluent by Proposition \ref{propolema:quasiaffine-and-stuff}(3).
The subspace $\mathcal{G}(X)$ is invariant for $T_{1}$, so that\[
\mathcal{G}(X)=\overline{m(T_{1})\mathcal{G}(v(T)^{-1}u(T))}\]
 for some inner function $m$. To prove the equality $X=v(T)^{-1}u(T)$,
it suffices to show that $m$ is in fact constant. Indeed, we have\begin{eqnarray*}
\overline{m(T)\mathcal{H}} & = & \overline{m(T)Q\mathcal{G}(v(T)^{-1}u(T))}=\overline{Qm(T_{1})\mathcal{G}(v(T)^{-1}u(T))}\\
 & = & \overline{Q\mathcal{G}(X)}=\overline{\mathcal{D}(X)}=\mathcal{H},\end{eqnarray*}
and the desired conclusion follows from the second assertion in Theorem
\ref{thm:spectral-picture-for-confluence}(4). There is no loss of
generality in assuming that $u$ and $v$ do not have any nonconstant
common inner divisor. We conclude the proof by showing that in this
case $v$ must be outer. Let $m$ be an inner divisor of $v$, and
note that for every $h\oplus k\in\mathcal{G}(X)$ we have \[
u(T)h=v(T)k\in\overline{m(T)\mathcal{H}},\]
and therefore $u(T)\mathcal{D}(X)\subset\overline{m(T)\mathcal{H}}.$
Since $\mathcal{D}(X)$ is dense in $\mathcal{H}$, we conclude that
$u(T_{(m(T)\mathcal{H})^{\perp}})=0$, and therefore $m$ divides
$u$. Thus $m$ is constant, and hence $v$ is outer.
\end{proof}
It follows from the results of \cite{C-D} that the one dimensional
spaces $\ker(\overline{\lambda}I-T)^{*}$ depend analytically on $\lambda$
and, in fact, there exists an analytic function $f:\mathbb{D}\to\mathcal{H}$
such that $\ker(\overline{\lambda}I-T)^{*}=\mathbb{C}f(\lambda)$
for $\lambda\in\mathbb{D}$. A local version of this result is easily
proved. Indeed, set $L=(T^{*}T)^{-1}T^{*}$. Given a unit vector $f_{0}\in\ker T^{*},$
the function\begin{equation}
f(\lambda)=(I-\lambda L^{*})^{-1}f_{0}=\sum_{k=0}^{\infty}\lambda^{n}L^{*n}f_{0}\label{eq:f(lambda)}\end{equation}
is analytic for $|\lambda|<1/\|L\|$, and obviously $T^{*}f(\lambda)=\lambda f(\lambda)$.
This calculation is valid for any left inverse of $T$. The operator
$L$ has the advantage that $L^{*}\mathcal{H}=T\mathcal{H}$, and
therefore $\langle T^{n}f_{0},f_{0}\rangle=\langle f_{0},L^{*n}f_{0}\rangle=0$
for $n\ge1$. These relations, along with $LT=I$, obviously imply\begin{equation}
\langle T^{n}f_{0},L^{*m}f_{0}\rangle=\delta_{nm},\quad n,m\in\mathbb{N}.\label{eq:bi-ortho}\end{equation}

\begin{prop}
\label{propo-former-lem:bi-orthogonal}Let $T\in\mathcal{L}(\mathcal{H})$
be a completely nonunitary contraction such that $H^{\infty}(T)$
is confluent. Define $L=(T^{*}T)^{-1}T^{*}$ and fix a unit vector
$f_{0}\in\ker T^{*}$.
\begin{enumerate}
\item The vector $f_{0}$ is cyclic for $L^{*}$.
\item $\bigcap\{T^{n}\mathcal{H}:n\ge0\}=\{0\}.$
\item $\bigcap\{L^{*n}\mathcal{H}:n\ge0\}=\mathcal{H}\ominus\left[\bigvee\{T^{n}f_{0}:n\ge0\}\right]$.
\end{enumerate}
\end{prop}
\begin{proof}
We have seen that $\ker(\lambda I-T)=\mathbb{C}f(\lambda)$ for $\lambda$
close to zero, where $f(\lambda)$ is given by (\ref{eq:f(lambda)})
and it belongs to $\bigvee\{L^{*n}f_{0}:n\ge0\}$. Thus (1) follows
from Theorem \ref{thm:spectral-picture-for-confluence}(3). To prove
$(2)$, let $h$ be an element in the intersection, and observe that
${\rm ord}_{0}(f/f_{0})\ge n$ for all $n\in\mathbb{N}$. Therefore
$f/f_{0}=0$, and necessarily $f=0$. The orthogonality relations
(\ref{eq:bi-ortho}) imply the inclusion\[
\bigcap_{n\ge0}L^{*n}\mathcal{H}\subset\mathcal{H}\ominus\bigvee_{n\ge0}T^{n}f_{0}.\]
Conversely, consider a vector $h\in\mathcal{H}\ominus\left[\bigvee\{T^{n}f_{0}:n\ge0\right]$.
Given $n\ge1$, we have\[
h=L^{*n}T^{*n}h+\sum_{k=0}^{n-1}L^{*k}(I-L^{*}T^{*})T^{*k}h.\]
Since $I-L^{*}T^{*}$ is the orthogonal projection onto $\mathbb{C}f_{0}$,
and \[
\langle T^{*k}h,f_{0}\rangle=\langle h,T^{k}f_{0}\rangle=0,\]
we deduce that $h=L^{*n}T^{*n}h\in R^{*n}\mathcal{H}$, thus proving
the opposite inclusion .
\end{proof}

\section{Confluence and Functional Models\label{sec:Confluence-and-Functional-mod}}

The results in Section \ref{sec:Confluent-Algebras-of} show that
completely nonunitary contractions $T$ for which $H^{\infty}(T)$
is confluent share many of the properties of the unilateral shift
$S$. In this section we will describe some quasiaffine transforms
of such operators $T$. These quasiaffine transforms are in fact functional
models associated with inner functions of the form\[
\Theta=\left[\begin{array}{c}
\theta_{1}\\
\theta_{2}\end{array}\right],\]
where $\theta_{1},\theta_{2}\in H^{\infty}$. The condition that $\Theta$
be inner amounts to the requirement that\[
|\theta_{1}(\zeta)|^{2}+|\theta_{2}(\zeta)|^{2}=1,\quad\text{{\rm a.e. }}\zeta\in\mathbb{T}.\]
We recall the construction of the functional model associated with
such a function $\Theta$. The subspace\[
\{\theta_{1}u\oplus\theta_{2}u:u\in H^{2}\}\subset H^{2}\oplus H^{2}\]
is obviously invariant for $S\oplus S$, and thus the orthogonal complement\[
\mathcal{H}(\Theta)=[H^{2}\oplus H^{2}]\ominus\{\theta_{1}u\oplus\theta_{2}u:u\in H^{2}\}\]
 is invariant for $S^{*}\oplus S^{*}$. The operator $S(\Theta)\in\mathcal{L}(\mathcal{H}(\Theta))$
is the compression of $S\oplus S$ to this space or, equivalently,
$S(\Theta)^{*}=(S^{*}\oplus S^{*})|\mathcal{H}(\Theta)$.
\begin{lem}
\label{lem:S(Theta)-confluent}Let $\Theta=\left[\begin{array}{c}
\theta_{1}\\
\theta_{2}\end{array}\right]$ be an inner function. The algebra $H^{\infty}(S(\Theta))$ is confluent
if and only if the functions $\theta_{1}$ and $\theta_{2}$ do not
have a nonconstant common inner factor.\end{lem}
\begin{proof}
If either of the functions $\theta_{j}$ is equal to zero, the other
one must be inner. The lemma is easily verified in this case. Indeed,
assume that $\theta_{1}$ is inner and $\theta_{2}=0$. If $\theta_{1}$
is not constant then $\ker\theta_{1}(S(\Theta))\ne\{0\}$, so that
$H^{\infty}(S(\Theta))$ is not confluent. Also, $\theta_{1}$ is
a common inner divisor of $\theta_{1}$ and $\theta_{2}$, so that
both conditions in the statement are false. On the other hand, if
$\theta_{1}$ is constant then $S(\Theta)$ is unitarily equivalent
to $S$, and the lemma is obvious in this case.

For the remainder of this proof, we consider the case in which both
functions $\theta_{j}$ are different from zero. Assume first that
$\theta_{j}=m\varphi_{j}$, where $m$ is a nonconstant inner function
and $\varphi_{j}\in H^{\infty}$ for $j=1,2$. The nonzero vector
$h\in\mathcal{H}(\Theta)$ defined by $h=P_{\mathcal{H}(\Theta)}(\varphi_{1}\oplus\varphi_{2})$
satisfies $m(S(\Theta))h=0$, and therefore the nonzero operator $m(S(\Theta))$
has nontrivial kernel. Thus $H^{\infty}(S(\Theta))$ is not confluent.

Assume now that $\theta_{1}$ and $\theta_{2}$ do not have a nonconstant
common inner factor. We verify first that $\ker u(S(\Theta))=\{0\}$
for $u\in H^{\infty}\setminus\{0\}$. It suffices to consider the
case of an inner function $u$. A vector $f_{1}\oplus f_{2}\in\ker u(S(\Theta))$
must satisfy $uf_{1}=\theta_{1}g$ and $uf_{2}=\theta_{2}g$ for some
$g\in H^{2}$. The fact that $\theta_{1}\wedge\theta_{2}=1$ implies
that $u$ divides $g$, and therefore $f_{1}\oplus f_{2}=\theta_{1}(g/u)\oplus\theta_{2}(g/u)$
belongs to $\mathcal{H}(\Theta)^{\perp}$ and the equality $f_{1}\oplus f_{2}=0$
follows. To conclude the proof, we will show that $h=P_{\mathcal{H}(\Theta)}(1\oplus0)$
is a rationally strictly cyclic vector for $H^{\infty}(S(\Theta))$.
Indeed, assume  that $f=f_{1}\oplus f_{2}\in\mathcal{H}(\Theta)\setminus\{0\}$,
and write $f_{1}=a_{1}/b$ and $f_{2}=a_{2}/b$, where $a_{1},a_{2},b\in H^{\infty}$
and $b$ is outer. Define functions $u=-b\theta_{2}$, $v=\theta_{1}a_{2}-\theta_{2}a_{1}$,
and note that\begin{eqnarray*}
v(S(\Theta))h-u(S(\Theta))f & = & P_{\mathcal{H}(\Theta)}(v\oplus0-uf_{1}\oplus uf_{2})\\
 & = & P_{\mathcal{H}(\Theta)}(\theta_{1}a_{2}\oplus\theta_{2}a_{2})=0.\end{eqnarray*}
The lemma follows because $u\not\equiv0$, and hence $u(S(\Theta))$
is injective.
\end{proof}
Let us remark that the condition $\theta_{1}\wedge\theta_{2}=1$ is
equivalent to the fact that the function $\Theta$ is $*$-outer.
In other words, the operators $S(\Theta)$ described in the preceding
lemma are of class $C_{10}$. This is in agreement with Theorem \ref{thm:spectral-picture-for-confluence}(6).
\begin{prop}
\label{pro:S(theta)-prec-confluenT}Assume that $T$ is a completely
nonunitary contraction such that $H^{\infty}(T)$ is confluent.
\begin{enumerate}
\item There exists an inner function $\Theta=\left[\begin{array}{c}
\theta_{1}\\
\theta_{2}\end{array}\right]$ such that $S(\Theta)\prec T$, and $H^{\infty}(S(\Theta))$ is confluent.
\item We have $S\prec T$ if and only if $T$ has a cyclic vector.
\end{enumerate}
\end{prop}
\begin{proof}
Denote by $U_{+}\in\mathcal{L}(\mathcal{K}_{+})$ the minimal isometric
dilation of $T$. Thus $\mathcal{H}\subset\mathcal{K}_{+}$ and $TP_{\mathcal{H}}=P_{\mathcal{H}}U_{+}$.
Since $T\in C_{10}$, the operator $U_{+}$ is a unilateral shift.
Let us set $\mathcal{M}=\bigvee\{T^{n}h_{1}:n\ge0\}$, where $h_{1}\in\mathcal{H}\setminus\{0\}$,
and let $h_{2}\in\mathcal{H}\ominus\mathcal{M}$ be a cyclic vector
for the compression of $T$ to this subspace. Such a vector exists
by Theorem \ref{thm:spectral-picture-for-confluence}(4). Observe
that $\mathcal{H}=\bigvee\{T^{n}h_{1},T^{n}h_{2}:n\ge0\}$. We define
now a space\[
\mathcal{E}=\bigvee\{U_{+}^{n}h_{1},U_{+}^{n}h_{2}:n\ge0\}\]
and an operator $Y\in\mathcal{L}(\mathcal{E},\mathcal{H})$ by setting
$Y=P_{\mathcal{H}}|\mathcal{E}$. The space $\mathcal{E}$ is invariant
for $U_{+}$, $Y(U_{+}|\mathcal{E})=TY$, and $Y$ has dense range.
Moreover, the restriction $U_{+}|\mathcal{E}$ is a unilateral shift
of multiplicity 1 or 2. Finally, set $\mathcal{H}'=\mathcal{E}\ominus\ker Y$,
$X=Y|\mathcal{H}'$, and denote by $T'$ the compression of $U_{+}$
to the space $\mathcal{H}'$. Then $X\mathcal{H}'=Y\mathcal{E}$ so
that $X$ is a quasiaffinity, and $XT'=TX$. Thus we have $T'\prec T$
and hence $H^{\infty}(T')$ is confluent by Proposition \ref{propolema:quasiaffine-and-stuff}(3).
 To conclude the proof, we need to show that $T'$ is unitarily equivalent
to an operator of the form $S(\Theta)$, where $\Theta=\left[\begin{array}{c}
\theta_{1}\\
\theta_{2}\end{array}\right]$ is an inner function. Equivalently, we must show that any compression
$T'$ of $S$ or of $S\oplus S$ to the orthogonal complement of an
invariant subspace is of this form provided that $H^{\infty}(T')$
is confluent. The compressions of $S$ are either $S$ itself, or
operators of the form $S(m)$. Among these only $S$ is confluent,
and it is of the form $S(\Theta)$ for $\Theta=\left[\begin{array}{c}
1\\
0\end{array}\right]$. The compressions of $S\oplus S$ are $S\oplus S$, $S(\Theta)$
with $\Theta$ an inner function of the desired form, or $S(\Theta)$
with $\Theta$ a $2\times2$ inner function. The compressions corresponding
to $2\times2$ matrices are operators of class $C_{0}$ (see \cite[Section VII.2]{Sz-N-F}),
and hence they do not generate confluent algebras. Finally, $H^{\infty}(S\oplus S)$
is not confluent as can be seen by considering the vectors $1\oplus0$
and $0\oplus1$.

If $T$ has a cyclic vector $h_{1}$, we can take $h_{2}=0$, and
then $U_{+}|\mathcal{E}$ is a shift of multiplicity 1. In this case,
we must have $\ker Y=\{0\}$ so that $U_{+}|\mathcal{E}\prec T$.
Conversely, $S\prec T$ implies that $T$ has a cyclic vector since
$S$ has one.
\end{proof}
The argument in the preceding proof appeared earlier in the classification
of contractions of class $C_{\cdot0}$ \cite{model-c-dot-0,injections},
and even earlier in \cite{Foias} and in the study of the class $C_{0}$
\cite{modele-de-j}.

When $T$ has a cyclic vector, it is natural to ask under what conditions
we actually have $T\sim S$.
\begin{lem}
\label{lem:cyclic-iff-restrictions-are}Assume that $T$ is a completely
nonunitary contraction such that $H^{\infty}(T)$ is confluent. The
following assertions are equivalent:
\begin{enumerate}
\item $T\prec S$.
\item $T|\mathcal{M}\prec S$ for some invariant subspace $\mathcal{M}$
of $T$.
\item $T|\mathcal{M}\prec S$ for every nonzero invariant subspace $\mathcal{M}$
of $T$.
\end{enumerate}
\end{lem}
\begin{proof}
The implications $(3)\Rightarrow(1)\Rightarrow(2)$ are obvious. Next
we show that $T\prec T|\mathcal{M}$ for every nonzero invariant subspace
$\mathcal{M}$ of $T$. By Theorem \ref{thm:spectral-picture-for-confluence}(4),
there is an inner function $m$ such that $\overline{m(T)\mathcal{H}}=\mathcal{M}$.
Then the operator $X:\mathcal{H}\to\mathcal{M}$ defined by $Xh=m(T)h$,
$h\in\mathcal{H},$ is a quasiaffinity and $XT=(T|\mathcal{M})X$.
Using this fact, it is easy to show that $(2)\Rightarrow(1)$. Indeed,
if (2) holds we have $T|\mathcal{M}\prec S$ for some $\mathcal{M}$,
and the relations $T\prec T|\mathcal{M}\prec S$ imply the desired
conclusion $T\prec S$. Finally, we prove that $(1)\Rightarrow(3)$.
Assume that (1) holds, so that $YT=SY$ for some quasiaffinity $Y$.
If $\mathcal{M}$ is a nonzero invariant subspace for $T$, the operator
$Z=Y|\mathcal{M}:\mathcal{M}\to\overline{Y\mathcal{M}}$ is a quasiaffinity
realizing the relation $T|\mathcal{M}\prec S|\overline{Y\mathcal{M}}$.
We conclude that (3) is true since $S|\overline{Y\mathcal{M}}$ is
unitarily equivalent to $S$.
\end{proof}
We can now state some conditions equivalent to the relation $T\sim S$.
\begin{thm}
\label{thm:TsimS-iff-exists-b}Assume that $T$ is a completely nonunitary
contraction such that $H^{\infty}(T)$ is confluent and it has a cyclic
vector. Let $f:\mathbb{D}\to\mathcal{H}$ be an analytic function
such that $\|f(0)\|=1$ and $\ker(\lambda I-T^{*})=\mathbb{C}f(\lambda)$
for every $\lambda\in\mathbb{D}$, and denote $\mathcal{H}_{0}=\bigvee\{T^{n}f(0):n\ge0\}$.
The following conditions are equivalent:
\begin{enumerate}
\item $T\sim S$.
\item $T|\mathcal{H}_{0}\prec S$.
\item There exists an outer function $b\in H^{\infty}$ such that $b(h/f(0))\in H^{2}$
for every $h\in\mathcal{H}_{0}$.
\item There exists an outer function $b\in H^{\infty}$ such that \[
b\frac{\langle h,f(\bar{\lambda})\rangle}{\langle f(0),f(\bar{\lambda})\rangle}\in H^{2}\]
for every $h\in\mathcal{H}_{0}$.
\end{enumerate}
\end{thm}
\begin{proof}
Since $T$ has a cyclic vector, we have $S\prec T$ by Proposition
\ref{pro:S(theta)-prec-confluenT}(2). Therefore $T\sim S$ is equivalent
to $T\prec S$, and this is equivalent to condition (2) by Lemma \ref{lem:cyclic-iff-restrictions-are}.
This establishes the equivalence $(1)\Leftrightarrow(2)$.

For an arbitrary $h\in\mathcal{H}\setminus\{0\}$, write the function
$h/f(0)$ as a quotient $u/v$ of functions in $H^{\infty}$. We have
then \[
\langle v(T)h,f(\bar{\lambda})\rangle=\langle h,v(T)^{*}f(\bar{\lambda})\rangle=\langle h,\overline{v(\lambda)}f(\bar{\lambda})\rangle=v(\lambda)\langle h,f(\bar{\lambda})\rangle,\]
and analogously $\langle u(T)f(0),f(\bar{\lambda})\rangle=u(\lambda)\langle f(0),f(\bar{\lambda})\rangle$.
Since $v(T)h=u(T)f(0)$, we conclude that\[
b(\lambda)\frac{h}{f(0)}(\lambda)=b(\lambda)\frac{\langle h,f(\bar{\lambda})\rangle}{\langle f(0),f(\bar{\lambda})\rangle}\]
for those $\lambda$ for which the denominators do not vanish. This
proves the equivalence $(3)\Leftrightarrow(4)$. Note that the analytic
function $\langle f(0),f(\bar{\lambda})\rangle$ cannot be identically
zero by Theorem \ref{thm:spectral-picture-for-confluence}(3).

It remains to prove the equivalence $(2)\Leftrightarrow(3)$, and
for this purpose we may as well assume that $\mathcal{H}=\mathcal{H}_{0}$.
We apply the construction in the proof of Proposition \ref{pro:S(theta)-prec-confluenT}
for this particular case. Thus, consider the minimal isometric dilation
$U_{+}\in\mathcal{K}_{+}$ of $T$, and denote $\mathcal{E}=\bigvee\{U_{+}^{n}f(0):n\ge0\}$.
Since $U_{+}^{*}f(0)=T^{*}f(0)=0$, there exists a unitary operator
$W:H^{2}\to\mathcal{E}$ such that $W1=f(0)$ and $WS=(U_{+}|\mathcal{E})W$.
One can then construct a quasiaffinity $Y:H^{2}\to\mathcal{H}$, namely
$Y=P_{\mathcal{H}}W$, such that $TY=YS$ and $Y1=f(0)$. Since an
equality of the form $v(S)x=u(S)1$ for $x\in H^{2}$ is equivalent
to $v(S)Yx=u(S)f(0)$, we deduce that \[
\frac{Yx}{f(0)}=x,\quad x\in H^{2},\]
and, conversely, any vector $h\in\mathcal{H}$ such that $k=h/f(0)\in H^{2}$
must belong to $YH^{2}$, namely $h=Yk$.

With this preparation, assume that $(2)$ holds, and let $X\in\mathcal{L}(\mathcal{H},H^{2})$
be a quasiaffinity such that $XT=SX$. Then the operator $XY$ is
a quasiaffinity in the commutant of $S$, and therefore $XY=b(S)$
for some outer function $b\in H^{\infty}$. The equality\[
X(b(T)-YX)=b(S)X-(XY)X=0\]
implies that we also have $YX=b(T)$. For any $h\in\mathcal{H}\setminus\{0\}$
we have then\[
b\frac{h}{f(0)}=\frac{b(T)h}{f(0)}=\frac{YXh}{Y1}=Xh\in H^{2},\]
thus proving (3). Conversely, if $(3)$ holds, we can define a linear
map $X:\mathcal{H}\to H^{2}$ by setting $Xh=b(h/f(0))$ for $h\in\mathcal{H}$,
and this map obviously satisfies $XT=SX$. It is easy to verify that
$X$ is a closed linear transformation and hence it is continuous.
It is also immediate that $XY=b(S)$ and $YX=b(T)$, and this implies
that $X$ is a quasiaffinity since $b$ is outer.\end{proof}
\begin{cor}
\label{cor:f(0)-vs-f(lamb)}Assume that $T\in\mathcal{L}(\mathcal{H})$
is a completely nonunitary contraction such that $T\sim S$. Let $f:\mathbb{D}\to\mathcal{H}$
be an analytic function such that $\|f(0)\|=1$ and $\ker(\lambda I-T^{*})=\mathbb{C}f(\lambda)$
for every $\lambda\in\mathbb{D}$, and assume that $\mathcal{H}=\bigvee\{T^{n}f(0):n\ge0\}$.
Then $\langle f(0),f(\lambda)\rangle\ne0$ for every $\lambda\in\mathbb{D}$.\end{cor}
\begin{proof}
Let $b$ be an outer function satisfying condition (4) of Theorem
\ref{thm:TsimS-iff-exists-b}. Assume that $\langle f(0),f(\bar{\lambda})\rangle=0$
for some $\lambda\in\mathbb{D}$. Since $b(\lambda)\ne0$, it follows
that $\langle h,f(\bar{\lambda})\rangle=0$ for every $h\in\mathcal{H}$,
and therefore $f(\bar{\lambda})=0$, which is impossible since this
vector generates $\ker(\lambda I-T)^{*}$.
\end{proof}
The conclusion of this corollary is not true for arbitrary contractions
$T$ for which $H^{\infty}(T)$ is confluent. Consider for instance
the function \[
\Theta=\left[\begin{array}{c}
\theta_{1}\\
\theta_{2}\end{array}\right],\]
where $\theta_{1}(z)=3z/5$ and $\theta_{2}(z)=4(2z-1)/(5(2-z))$
for $z\in\mathbb{D}$. We have $\theta_{1}(S)^{*}1=0$ and $\theta_{2}(S)^{*}x=x/2$
with $x(z)=1/(2-z)$, $z\in\mathbb{D}$. It follows easily that $\ker S(\Theta)^{*}$
is generated by $1\oplus0$, while $\ker(\frac{1}{2}I-S(\Theta))^{*}$
is generated by $0\oplus x$. For this example we have therefore $\langle f(0),f(\frac{1}{2})\rangle=0$.
According to the preceding corollary, $\mathcal{H}_{0}=\bigvee\{S(\Theta)^{n}f(0):n\ge0\}$
must be a proper subspace of $\mathcal{H}(\Theta)$. It is easy to
verify that $\mathcal{H}(\Theta)\ominus\mathcal{H}_{0}$ is precisely
$\mathbb{C}f(\frac{1}{2})$.

The relation $T\prec S$ can also be studied in terms of the minimal
unitary dilation of $T$. We will denote by $R_{*}\in\mathcal{L}(\mathcal{R}_{*})$
the $*$-residual part of this minimal unitary dilation; see \cite[Section II.3]{Sz-N-F}
for the relevant definitions. Note however that our $R_{*}$ is the
adjoint of the one considered there. The facts we require about this
operator are as follows:
\begin{enumerate}
\item [(a)]$R_{*}$ is a unitary operator with absolutely continuous spectral
measure relative to arclength measure on $\mathbb{T}$.
\item [(b)]There exists an operator $Z:\mathcal{H}\to\mathcal{R}_{*}$
(namely, the orthogonal projection onto $\mathcal{R}_{*}$) such that
$ZT=R_{*}Z$ and \[
\|Zh\|=\lim_{n\to\infty}\|T^{n}h\|.\]
In particular, $Z$ is injective if and only if $T$ is of class $C_{1\cdot}$.
\item [(c)]The smallest reducing subspace for $R_{*}$ containing $Z\mathcal{H}$
is $\mathcal{R}_{*}$.\end{enumerate}
\begin{prop}
\label{pro:conf-implies-R*cyclic}Assume that $T$ is a completely
nonunitary contraction such that $H^{\infty}(T)$ is confluent.
\begin{enumerate}
\item The $*$-residual part $R_{*}$ of the minimal unitary dilation of
$T$ has spectral multiplicity at most $1$.
\item We have $T\prec S$ if and only if $R_{*}$ is a bilateral shift of
multiplicity $1$.
\item We have $T\prec R_{*}|\overline{Z\mathcal{H}}$, and $T\prec S$ if
and only if $\overline{Z\mathcal{H}}\ne\mathcal{R}_{*}$.
\item $T^{*}$ has a cyclic vector.
\end{enumerate}
\end{prop}
\begin{proof}
Given $h_{1},h_{2}\in\mathcal{H}\setminus\{0\}$, select $u_{1},u_{2}\in H^{\infty}\setminus\{0\}$
such that $u_{1}(T)h_{1}=u_{1}(T)h_{2}$. Then we have $u_{1}(R_{*})Zh_{1}=u_{2}(R_{*})Zh_{2}$.
Since $u_{1}(\zeta)$ and $u_{2}(\zeta)$ are different from zero
a.e. relative to the spectral measure of $R_{*}$, it follows that
the vectors $Zh_{1}$ and $Zh_{2}$ generate the same reducing space
for $R_{*}$. Therefore $R_{*}$ has a $*$-cyclic vector, and this
implies (1).

Next we prove (3). The fact that $T\prec R_{*}|\overline{Z\mathcal{H}}$
is immediate. If $\overline{Z\mathcal{H}}$ is not reducing, then
$R_{*}|\overline{Z\mathcal{H}}$ is unitarily equivalent to $S$ and
hence $T\prec S$. Conversely, if $T\prec S$, let $W$ be a quasiaffinity
such that $WX=SW$. For any $h\in\mathcal{H}$ we have \[
\|Wh\|=\lim_{n\to\infty}\|S^{n}Wh\|=\lim_{n\to\infty}\|WT^{n}h\|\le\|W\|\|Zh\|,\]
so that there exists an operator $X:\overline{Z\mathcal{H}}\to H^{2}$
such that $\|X\|\le\|W\|$ and $XZ=W$. Since the range of $X$ contains
the range of $W$, we have $X\ne0$. Pick a vector $f\in H^{2}$ such
that $X^{*}f\ne0$, and observe that\[
\lim_{n\to\infty}\|(R_{*}|\overline{Z\mathcal{H}})^{*n}X^{*}f\|=\lim_{n\to\infty}\|X^{*}S^{*n}f\|=0.\]
Therefore $R_{*}|\overline{Z\mathcal{H}}$ is not unitary, and consequently
$\overline{Z\mathcal{H}}\ne\mathcal{R}_{*}$.

Assume now that $T\prec S$. The fact that $R_{*}$ is a bilateral
shift follows from (3) because the only unitary operator of multiplicity
$1$ which has nonreducing invariant subspaces is the bilateral shift.
Conversely, if $R_{*}$ is a bilateral shift, the results of \cite{kerchy-shifts}
imply the existence of an invariant subspace $\mathcal{M}$ for $T$
such that $T|\mathcal{M}\prec S$. We deduce that $T\prec S$ by Lemma
\ref{lem:cyclic-iff-restrictions-are}. This proves (2).

Finally, (4) also follows from (3) because $(R_{*}|\overline{Z\mathcal{H}})^{*}$
has a cyclic vector.\end{proof}
\begin{cor}
\label{cor:S(theta)=00003Dprec-S}Assume that $\Theta=\left[\begin{array}{c}
\theta_{1}\\
\theta_{2}\end{array}\right]$ is inner and $*$-outer. Then $S(\Theta)\prec S$. More precisely,
the operator $Q:\mathcal{H}(\Theta)\to H^{2}$ defined by $Q(f_{1}\oplus f_{2})=\theta_{1}f_{2}-\theta_{2}f_{1}$,
$f_{1}\oplus f_{2}\in\mathcal{H}(\Theta)$, is a quasiaffinity and
$QS(\Theta)=SQ$.

More generally, we have $T\prec S$ whenever $T$ is a completely
nonunitary contraction, $H^{\infty}(T)$ is confluent, and $I-T^{*}T$
has finite rank.\end{cor}
\begin{proof}
We will show that $\overline{P_{\mathcal{R}_{*}}\mathcal{H}(\Theta)}\ne\mathcal{R}_{*}$.
To do this, we observe first that the minimal unitary dilation of
$S(\Theta)$ is the operator $U\oplus U$ on $L^{2}\oplus L^{2}$.
The space $\mathcal{R}_{*}$ is the orthogonal complement of the smallest
reducing space for $U\oplus U$ containing $\{\theta_{1}u\oplus\theta_{2}u:u\in H^{2}\}$.
Thus \[
\mathcal{R}_{*}=(L^{2}\oplus L^{2})\ominus\{\theta_{1}u\oplus\theta_{2}u:u\in L^{2}\},\]
 and it follows that $P_{\mathcal{R}_{*}}$ is the operator of pointwise
multiplication by the matrix\[
I-\Theta\Theta^{*}=\left[\begin{array}{cc}
|\theta_{2}|^{2} & -\overline{\theta_{2}}\theta_{1}\\
-\theta_{2}\overline{\theta_{1}} & |\theta_{1}|^{2}\end{array}\right].\]
Finally, we have $P_{\mathcal{R}_{*}}\mathcal{H}(\Theta)=P_{\mathcal{R}_{*}}(H^{2}\oplus H^{2})$,
and therefore $\overline{P_{\mathcal{R}_{*}}\mathcal{H}(\Theta)}$
is the invariant subspace for $U$ generated by $P_{\mathcal{R}_{*}}(1\oplus0)$
and $P_{\mathcal{R}_{*}}(0\oplus1)$. These two vectors are precisely
\begin{eqnarray*}
|\theta_{2}|^{2}\oplus(-\theta_{2}\overline{\theta_{1}}) & = & (-\overline{\theta_{2}}u)\oplus\overline{\theta_{1}}u,\\
(-\overline{\theta_{2}}\theta_{1})\oplus|\theta_{1}|^{2} & = & (-\overline{\theta_{2}}v)\oplus\overline{\theta_{1}}v,\end{eqnarray*}
 with $u=-\theta_{2}$ and $v=\theta_{1}$. Since $\theta_{1}$ and
$\theta_{2}$ do not have nonconstant common inner divisors, the invariant
subspace for $S$ they generate is the entire $H^{2}$. It follows
that\[
\overline{P_{\mathcal{R}_{*}}\mathcal{H}(\Theta)}=\{(-\overline{\theta_{2}}u)\oplus\overline{\theta_{1}}u:u\in H^{2}\},\]
and $R_{*}|\overline{P_{\mathcal{R}_{*}}\mathcal{H}(\Theta)}$ is
unitarily equivalent to $S$. The final assertion is verified by noting
that \[
P_{\mathcal{R}_{*}}(f_{1}\oplus f_{2})=(-\overline{\theta_{2}}Q(f_{1}\oplus f_{2}))\oplus(\overline{\theta_{1}}Q(f_{1}\oplus f_{2}))\]
for $f_{1}\oplus f_{2}\in\mathcal{H}(\Theta)$.

To verify the last assertion, denote by $n$ the rank of $I-T^{*}T$,
and observe that the characteristic function $\Theta_{T}$ is inner,
$*$-outer, and it coincides with an $(n+1)\times n$ matrix over
$H^{\infty}$. Indeed, $\Theta_{T}(0)$ is a Fredholm operator of
index $-1$. It follows that $I-\Theta_{T}(\zeta)\Theta_{T}(\zeta)^{*}$
has rank $1$ for a.e. $\zeta\in\mathbb{T}$, and therefore $\mathcal{R}_{*}$
is a bilateral shift by \cite[Section VI.6]{Sz-N-F}. The result follows
now from \ref{pro:conf-implies-R*cyclic}(2).\end{proof}
\begin{cor}
\label{cor:S(theta)not-cyclic}Assume that $\Theta=\left[\begin{array}{c}
\theta_{1}\\
\theta_{2}\end{array}\right]$ is inner and $*$-outer.
\begin{enumerate}
\item If $f_{1}\oplus f_{2}\in\mathcal{H}(\Theta)$ is cyclic for $S(\Theta)$,
then $\theta_{1}f_{2}-\theta_{2}f_{1}$ is an outer function.
\item Conversely, if $\theta_{1}f_{2}-\theta_{2}f_{1}$ is outer for some
$f_{1},f_{2}\in H^{2}$, then $P_{\mathcal{H}(\Theta)}(f_{1}\oplus f_{2})$
is cyclic for $S(\Theta)$.
\item There exists $\Theta$ such that $S(\Theta)$ does not have a cyclic
vector.
\item We have $S(\Theta)\sim S$ if and only if $S(\Theta)$ has a cyclic
vector.
\end{enumerate}
\end{cor}
\begin{proof}
With the notation of Corollary \ref{cor:S(theta)=00003Dprec-S}, $Q(f_{1}\oplus f_{2})$
must be cyclic for $S$ if $f_{1}\oplus f_{2}$ is cyclic for $S(\Theta)$.
This proves (1).

Conversely, assume that $u=\theta_{1}f_{2}-\theta_{2}f_{1}$ is outer
for some $f_{1},f_{2}\in H^{2}$. Upon multiplying $f_{1},f_{2}$
by some outer function, we may assume that $f_{1},f_{2}\in H^{\infty}$.
Let $g_{1}\oplus g_{2}\in\mathcal{H}(\Theta)$ be a vector orthogonal
to $\bigvee\{S(\Theta)^{n}P_{\mathcal{H}(\Theta)}(f_{1}\oplus f_{2}):n\ge0\}$.
We have then\[
\langle g_{1}\oplus g_{2},\theta_{1}p\oplus\theta_{2}p\rangle=\langle g_{1}\oplus g_{2},f_{1}p\oplus f_{2}p\rangle=0\]
for every polynomial $p$. Equivalently, $\overline{\theta_{1}}g_{1}+\overline{\theta_{2}}g_{2}$
and $\overline{f_{1}}g_{1}+\overline{f_{2}}g_{2}$ belong to $L^{2}\ominus H^{2}$,
and therefore the functions\begin{eqnarray*}
\overline{u}g_{1} & = & \overline{f_{2}}(\overline{\theta_{1}}g_{1}+\overline{\theta_{2}}g_{2})-\overline{\theta_{2}}(\overline{f_{1}}g_{1}+\overline{f_{2}}g_{2}),\\
\overline{u}g_{2} & = & \overline{\theta_{1}}(\overline{f_{1}}g_{1}+\overline{f_{2}}g_{2})-\overline{f_{1}}(\overline{\theta_{1}}g_{1}+\overline{\theta_{2}}g_{2})\end{eqnarray*}
are also in $L^{2}\ominus H^{2}$. Thus $\langle g_{j},up\rangle=0$
for all polynomials $p$, and hence $g_{j}=0$, $j=1,2$, because
$u$ is outer. Assertion (2) follows.

To prove (3), let $m_{1}$ and $m_{2}$ be two relatively prime inner
functions, and set $\theta_{1}=\frac{3}{5}m_{1}$ and $\theta_{2}=\frac{4}{5}m_{2}$.
Nordgren \cite{nordgren} showed that it is possible to choose $m_{1}$
and $m_{2}$ so that no function of the form $m_{1}f_{2}-m_{2}f_{1}$
is outer if $f_{1},f_{2}\in H^{2}$. The corresponding operator $S(\Theta)$
does not have a cyclic vector. Finally (4) follows from Corollary
\ref{cor:S(theta)=00003Dprec-S} and Proposition \ref{pro:S(theta)-prec-confluenT}(2).
\end{proof}
Let us also note a related result which follows easily from \cite{Sn-F-similarity}.
\begin{prop}
\label{pro:S(theta)-sim-toS}Assume that $\Theta=\left[\begin{array}{c}
\theta_{1}\\
\theta_{2}\end{array}\right]$ is inner and $*$-outer. The operator $S(\Theta)$ is similar to
$S$ if and only if there exist $f_{1},f_{2}\in H^{\infty}$ such
that $\theta_{1}f_{2}-\theta_{2}f_{1}=1$.\end{prop}
\begin{proof}
It was shown in \cite{Sn-F-similarity} that $S(\Theta)$ is similar
to an isometry if and only if $\Theta$ is left invertible. To conclude,
one must observe that the only possible isometry is a unilateral shift
of multiplicity 1.
\end{proof}
Some of the statements of Proposition \ref{pro:conf-implies-R*cyclic}
remain valid for arbitrary completely nonunitary contractions. The
proof of the following proposition follows easily from the above arguments,
along with the corresponding properties of $S$.
\begin{prop}
\label{pro:arbitrary-qsim-S}Let $T$ be a completely nonunitary contraction
such that $T\sim S$. Then $T$ is of class $C_{10}$, both $T$ and
$T^{*}$ have cyclic vectors, and the $*$-residual part $R_{*}$
of the minimal unitary dilation of $T$ is a bilateral shift of multiplicity
$1$.
\end{prop}
The converse of this proposition is not true. Indeed, it was shown
in \cite{berker} (see also \cite{ker}) that there exist operators
$T$ of class $C_{10}$, with a cyclic vector, such that $R_{*}$
is a bilateral shift of multiplicity $1$, and $\sigma(T)\not\supset\mathbb{D}$.
For such operators we will have $R_{*}^{*}\prec T^{*}$, so that $T^{*}$
also has a cyclic vector, but $T\not\prec S$.

The converse does however hold provided that $H^{\infty}(T)$ is confluent.
This follows from Propositions \ref{pro:S(theta)-prec-confluenT}(2)
and \ref{pro:conf-implies-R*cyclic}(2).


\begin{thebibliography}{21}
\bibitem{A-D} M. B. Abrahamse and R. G. Douglas, A class of subnormal
operators related to multiply connected domains, \emph{Adv. Math.}
\textbf{19} (1976), 106--148.

\bibitem{arv}W. B. Arveson, A density theorem for operator algebras,
\emph{Duke Math. J.} \textbf{34} (1967), 635--647.

\bibitem{Berco} H. Bercovici, \emph{Operator Theory and Arithmetic
on $H^{\infty}$}, American Mathematical Society, Providence, RI,
1988.

\bibitem{B-F-P} H. Bercovici, C. Foias, and C. Pearcy, \emph{Dual
algebras with applications to invariant subspaces and dilation theory},
American Mathematical Society, Providence, RI, 1985.

\bibitem{berker}H. Bercovici and L. K\'erchy, Spectral behaviour
of $C_{10}$-contractions, preprint.

\bibitem{brown}S. W. Brown, Contractions with spectral boundary,
\emph{Integral Equations Operator Theory} \textbf{11} (1988), 49--63.

\bibitem{bcp}S. W. Brown, B. Chevreau, and C. Pearcy, On the structure
of contraction operators. II, \emph{J. Funct. Anal.} \textbf{76} (1988),
no. 1, 30--55.

\bibitem{C-D} M. J. Cowen and R. G. Douglas, Complex geometry and
operator theory, \emph{Acta Math.} \textbf{141 (}1978), 187--261.

\bibitem{fish}S. D. Fisher, \emph{Function theory on planar domains.
A second course in complex analysis}, John Wiley \& Sons, New York,
1983.

\bibitem{Foias} C. Foias, A classification of doubly cyclic operators,
\emph{Colloquia Math. Soc. J. Bolyai.} \textbf{5}, Hilbert Space Operators,
Tihany, 1970, 155--161.

\bibitem{ker}L. K\'erchy, On the spectra of contractions belonging
to special classes, \emph{J. Funct. Anal.} \textbf{67} (1986), 153--166.

\bibitem{kerchy-shifts}---------, Shift-type invariant subspaces
of contractions, \emph{J. Funct. Anal.} \textbf{246} (2007), 281--301.

\bibitem{nordgren}E. A. Nordgren, The ring $N_{+}$ is not adequate,
\emph{Acta Sci. Math. (Szeged)} \textbf{36} (1974), 203--204.

\bibitem{rad-ros}H. Radjavi and P. Rosenthal, \emph{Invariant subspaces.
Second edition}, Dover Publications, Mineola, NY, 2003.

\bibitem{Sz-N-F} B. Sz.-Nagy and C. Foias, \emph{Harmonic Analysis
of Operators on Hilbert Space}, North Holland, Amsterdam, 1970.

\bibitem{modele-de-j}---------, Mod\`ele de Jordan pour une classe
d'op\'erateurs de l'espace de Hilbert, \emph{Acta Sci. Math. (Szeged)}
\textbf{31} (1970), 91--115.

\bibitem{local-C_0}---------, Local characterization of operators
of class $C_{0}$, \emph{J. Funct. Anal.} \textbf{8} (1971), 76--81.

\bibitem{complements2}---------, Compl\'ements à l'\'etude des
op\'erateurs de classe $C_{0}$. II, \emph{Acta Sci. Math. (Szeged)}
\textbf{33} (1972), 113--116.

\bibitem{model-c-dot-0}---------, Jordan model for contractions of
class $C_{\cdot0}$, \emph{Acta Sci. Math. (Szeged)} \textbf{36} (1974),
305--322.

\bibitem{injections} ---------, Injections of shifts into strict
contractions, \emph{Linear Operators and Approximation. II}, Birkh\"auser,
Basel, 1974, 29--37.

\bibitem{Sn-F-similarity}---------, On contractions similar to isometries
and Toeplitz operators, \emph{Ann. Acad. Sci. Fenn. Ser. A I Math.}
\textbf{2} (1976), 553--564.
\end{thebibliography}
\end{document}